\documentclass[a4paper,reqno]{amsart}
\usepackage{dsfont}
\usepackage{newcent} 
\usepackage[T1]{fontenc}
\usepackage{amsmath,amssymb,amsfonts,amsthm}
\usepackage{mathtools}
\usepackage{newlfont}
\usepackage{fancyhdr,fancyvrb}
\usepackage{graphicx}
\usepackage{nextpage}
\usepackage[dvipsnames,usenames]{color}
\usepackage[colorlinks=true, linkcolor=blue, citecolor=ForestGreen]{hyperref}

\newcommand{\sfd}{\mathsf d}
\newcommand{\crd}{}

\newcommand{\N}{\mathbb{N}}
\newcommand{\R}[1]{\mathbb{R}^{#1}}

\newcommand{\bb}{\boldsymbol{b}}

\newcommand{\by}{\boldsymbol{y}}
\newcommand{\bY}{\mathbf{Y}}

\newcommand{\cF}{\mathcal F}

\newcommand{\cK}{\mathcal K}
\newcommand{\cL}{\mathcal L}

\newcommand{\cP}{\mathcal P}

\newcommand{\cS}{\mathcal S}
\newcommand{\cT}{\mathcal T}

\newcommand{\cW}{\mathcal W}

\newcommand{\eeta}{\boldsymbol\eta}

\newcommand{\BL}{\mathrm{BL}}

\newcommand{\de}{\mathrm d}

\newcommand{\spann}{\operatorname{span}}

\newcommand{\diam}{\operatorname{diam}}
\newcommand{\dist}[2]{\operatorname{dist}(#1,#2)}

\newcommand{\Lip}{\operatorname{Lip}}

\newcommand{\supp}{\operatorname{supp}}

\DeclareMathOperator*{\ev}{ev}

\newcommand{\blu}[1]{\textcolor[rgb]{0,0,1}{#1}}

\renewcommand{\div}{\operatorname{div}}
\renewcommand{\geq}{\geqslant}
\renewcommand{\leq}{\leqslant}

\newcommand{\average}{{\mathchoice {\kern1ex\vcenter{\hrule
height.4pt width 8pt depth0pt}
\kern-11pt} {\kern1ex\vcenter{\hrule height.4pt width 4.3pt
depth0pt} \kern-7pt} {} {} }}

\newcommand{\be}{\begin{equation}}
\newcommand{\ee}{\end{equation}}

\renewcommand{\vec}[2]{\left(\begin{array}{c}
                    #1 \\
                    #2
                    \end{array}\right)}

\mathchardef\emptyset="001F

\numberwithin{equation}{section}

\newtheorem{defin}{Definition}[section]
\newtheorem{remark}[defin]{Remark}
\newtheorem{theorem}[defin]{Theorem}
\newtheorem{lemma}[defin]{Lemma}
\newtheorem{proposition}[defin]{Proposition}
\newtheorem{cor}[defin]{Corollary}

\title[Mean-field analysis of multi-population dynamics with label switching]{Mean-field analysis of multi-population dynamics \\ with label switching}%
\author{}
\address{}
\email{}
\author{Marco Morandotti}
\address{Dipartimento di Scienze Matematiche ``G.~L.~Lagrange'', Politecnico di Torino, Corso Duca degli Abruzzi, 24, 10129 Torino, Italy}
\email[M.~Morandotti]{marco.morandotti@polito.it}
\author{Francesco Solombrino}
\address{Dipartimento di Matematica e Applicazioni ``Renato Caccioppoli'', Universit\`a degli Studi di Napoli Federico II, Via Cintia, Monte S.~Angelo, 80126 Napoli, Italy}
\email[F.~Solombrino]{francesco.solombrino@unina.it}

\date{\today}


\makeindex

\begin{document}

\begin{abstract}
The mean-field analysis of a multi-population agent-based model is performed. The model couples a particle dynamics driven by a nonlocal velocity with a Markow-type jump process on the probability that each agent has of belonging to a given population. A general functional analytic framework for the well-posedness of the problem is established, and some concrete applications are presented, both in the case of discrete and continuous set of labels. In the particular case of a leader-follower dynamics, the existence and approximation results recently obtained in \cite{ABRS2018} are  recovered and generalized as a byproduct of the abstract approach proposed.
\end{abstract}

\maketitle


\tableofcontents


\section{Introduction} \label{sec:intro}
The concept of mean-field interaction, originally used in statistical physics with Kac \cite{Kac} and then
McKean \cite{mckean1967propagation}  to describe the collisions between particles in a gas, has later
proved to be a powerful tool to analyze the asymptotic behavior of systems of interacting agents. Recent applications range from biological, social and economical phenomena \cite{albi2015invisible,parisi08,carrillo2014derivation} to automatic learning \cite{maggioni17, huang2018learning} and optimization heuristics \cite{dorigo2005ant,kennedy2011particle}. The underlying idea is that the collective behavior of large systems of particles (agent-based models, usually consisting of a set of ODE's) can be efficiently treated by replacing the influence of all the other individuals in the dynamics on a given agent by a single averaged effect. From a mathematical point of view, this amounts to passing from a particle description to a kinetic description, consisting of a limit PDE whose unknown is the particle density distribution in the state space. The well-posedness of such  models has therefore to be proven in spaces of measures (see, for instance, the results in \cite{CCR2011}).

In some of the applications, the interacting agents are assumed to belong to a number of different species, or populations \cite{albi2014boltzmann,albi2017opinion,cirant2015multi,francesco2013measure,fornasier2014mean}. This is a useful modeling assumption, \emph{e.g.}, in the theory of mean-field games, or in control theory, where it can been used to distinguish informed agents steering pedestrians to leave unknown environments,  or to highlight the influence of few key investors in the stock market (see the discussion in \cite[Introduction]{bongini2016optimal}). In a multi-population setting,  source (or sink) terms can be added to the model, to account for the ''birth'' and ''death'' rate of a single population (see for instance \cite[Sections 4-5]{Rossi-review}). Furthermore, even in the case where the  number of agents remains constant along the evolution, the possibility of having some individuals switching from one population to another has been considered \cite{thai}. 

As a matter of fact, exchange rates among populations are a common feature in several applications. 
For instance, in models of chemical reaction networks a particle may change its type as a result of the interaction with the others, at a stochastic rate which may also depend on its position \cite{LLN2019,CRN,O1989}.
Another relevant example comes from social dynamics, where transition rates appear in some kinetic models of opinion formation in the presence of strong leaders, as the one proposed in \cite{BMPW2009}, inspired by the earlier contribution \cite{Toscani2006}.
In this model, the opinion of the agents, described by their position in the state space, evolves because of the exchange of opinion with the others.
This is encoded by a nonlocal transport term taking into account the presence, among the overall population, of a restricted number of ``leaders'' promoting their opinion with a strong influence on the ``followers''.
It is natural to postulate, as the authors do (see \cite[Section~3.b]{BMPW2009} for a detailed discussion), that opinion leadership is not constant over time: someone who is an opinion leader today may lose this role tomorrow, or a follower may become a leader in the future.

Summarizing, the multi-population dynamics we are interested in attaches to an agent sitting at a position $x$ a probability measure $\lambda\in\cP(U)$, where $U$ is a compact space of labels accounting for the population to which the agent belongs.
In fact, we find it natural not to assume any deterministic knowledge of the label of a single agent, since transitions are  usually modeled as the outcome of a stochastic process. Hence, $\lambda$ expresses the probability that the agent has, at a certain moment in time, of belonging to a subset of $U$, and may itself evolve as a consequence of the interaction among the agents.
The minimal model for evolution that we propose\footnote{We neglect in the present contribution the possible presence of diffusion terms.} consists of the coupling of a non-local transport dynamics with a Markov-type jump process. 
We namely assume that
\begin{itemize}
\item a particle at position $x_i$ with probability distribution of labels $\lambda_1$ experiences a velocity field $\dot x_i=v_{\Psi^N}(x_i,\lambda_i)$ influenced by the global state of the system $\Psi^N$, the empirical measure $\frac1N\sum_{i=1}^N \delta_{(x_i,\lambda_i)}$;
\item the probability distribution of labels $\lambda_i$ evolves according to $\dot \lambda_i=\cT^*(x_i,\Psi^N)\lambda_i$, where the operator $\cT^*$, accounting for the transitions among the labels, is the adjoint of an operator $\cT(x,\Psi)$ defined on Lipschitz functions on $U$ and may depend on the position of the particle and on the state of the system.
\end{itemize}
In the case of two populations $U=\{F,L\}=\{\text{followers},\text{leaders}\}$, a model of this type has been recently analyzed in \cite{ABRS2018}, as a (partial) discrete counterpart of the PDE model in \cite{BMPW2009}.
Some simplifying assumptions had however to be added in order to perform a mean-field analysis.
In particular, the velocity field
\begin{subequations}\label{I11}
\begin{equation}\label{I1}
v_{\Psi^N}(x_i)= \frac{1}{N}\sum_{j=1}^N K^F(x_i-x_j)\lambda_j(\{F\})+\frac{1}{N}\sum_{j=1}^N K^L(x_i-x_j)\lambda_j(\{L\})
\end{equation}
did not depend on the probability on the labels $\lambda_i$ and the transition operators 
\begin{equation}\label{I2}
\begin{split}
(\cT^*(\Psi^N)\lambda_i)(\{F\})&=-\alpha_F (\Psi^N)\lambda_i(\{F\}) + \alpha_L(\Psi^N)(1-\lambda_i(\{F\}))\,,\\
(\cT^*(\Psi^N)\lambda_i)(\{L\})&=\alpha_F (\Psi^N)\lambda_i(\{F\}) - \alpha_L(\Psi^N)(1-\lambda_i(\{F\}))
\end{split}
\end{equation}
\end{subequations}
did not depend on the position $x_i$.
The goal of the present paper is to get rid of these simplifying assumptions, providing an \emph{appropriate functional setting} as well as a \emph{general set of assumptions} on the velocity field $v_\Psi$ and on the operator $\cT^*$ (see Section~\ref{sec:abstract_mod}) which allow one to perform a rigorous mean-field analysis of agent-based models of the kind discussed above.

In order to develop our analysis, we will borrow some functional-analytic and some measure-theoretic tools that have been recently specified for the different, although related, context of \emph{spatially inhomogeneous evolutionary games} in \cite{AFMS2018}.
The analogue of \eqref{I11} is described in \cite[Section~3.2]{AFMS2018}, where the dynamics considered is
\begin{equation}\label{I3}
     \left\{
    \begin{aligned}
      \dot x_{i}&=a(x_{i},\sigma_{i})=\int_U e(x_{i},u)\,\de\sigma_{i}(u),\\
      \dot \sigma_{i}&=
      \bigg(\frac1N\sum_{j=1}^N
      \Big(\int _U J(x_{i},\cdot,x_{j},u')\,\de\sigma_{j}(u')-
      \int_U\int_U
      J(x_{i},w,x_{j},u')\,\de\sigma_{j}(u')\,\de\sigma_{i}(w)\Big)\bigg)
      \sigma_{i},
    \end{aligned}
    \right.
\end{equation}
for $x\in\R{d}$ tracking the position of a player and $\sigma\in\cP(U)$ denoting their mixed strategy. 
In \eqref{I3}, the velocity $\dot x_i$ is determined by the $x$-dependent vector field $e(x,u)$ through an averaging process with respect to the strategies available to the player $x_i$; the evolution of the mixed strategy $\sigma_i$ is dictated by the \emph{replicator equation} (see \cite{HS1998}): the term in the parenthesis in the right-hand side of the second equation above determines the performance of the strategy played by player $x_i$ with respect to all of the available strategies, averaging out all of the possible distributions of the opponents $x_j$ with their mixed strategies $\sigma_j$.
In particular, $J(x_i,u,x_j,u')$ is the payoff of the game between player $x_j$ with strategy $u$ against player $x_j$ with strategy $u'$.
For the heuristic interpretation and the modeling issues related to \eqref{I3} we refer the interested reader to \cite{BorgersSarin,CH2017}.

The paper \cite{AFMS2018} introduced suitable notions (namely, Lagrangian and Eulerian) of solutions to the mean-field limit of \eqref{I3} showing convergence of the particle model to the limit description by means of stability estimates in the Wasserstein metric arising from the well-posedness theory of ODE's in Banach spaces \cite{Brezis,Cartan,Dieudonne}.

While sharing some common features, the dynamics described by \eqref{I11} and \eqref{I3} show a relevant difference in the structures of the velocity fields in the right-hand sides.
Indeed, in our setting, it is natural to postulate that the velocity of the agents is also depending on the behavior of the ones around them, at least in a small neighborhood. 
Such a feature is not present in the first equation in \eqref{I3}, whereas it is encoded in our model by considering a velocity field $v_\Psi$ depending on the global state of the system.
Beside this main distinction, also the right-hand sides of the equations for $\sigma_i$ and $\lambda_i$ present some differences; we refer the reader to Remark~\ref{differences} for a more detailed comparison of the two models.

The novel contribution of this work is in specifying the tools from \cite{AFMS2018} to the case of our particle system \eqref{I11}.
In doing this, we will pursue a more abstract point of view than the one considered in \cite{AFMS2018}, by focusing on the structural assumptions on the velocity field $v_\Psi$ and on the operator $\cT$ that guarantee the well-posedness of the mean-field analysis.
However, we will show in Sections~\ref{sec:LF} and~\ref{sec:other} that velocities and transition rates modeled by interaction kernels as those considered in \eqref{I11} and \eqref{I3} (in this last case, with the due difference that we pointed out before) also fit in the setting that we propose in Section~\ref{sec:abstract_mod}.
Thus, we can retrieve an even more general version of the results of \cite{ABRS2018} as a by-product of our analysis (see Theorem~\ref{T187}, where also an explicit dependence on the space variable $x$ of the transition rates is allowed).

The general assumptions on $v_\Psi$ and $\cT$ are presented at the beginning of Section~\ref{sec:abstract_mod}, see (v1)-(v3) and (T0)-(T3) below.
The dependence of $v_\Psi$ on the global status of the system call for a locally Lipschitz dependence both on the status variable $(x,\lambda)$ (in a suitable topology) and on $\Psi$ with respect to the Wasserstein distance between probability measures, which is customary for the mean-field analysis of transport equation with non-local vector fields, see
e.g. \cite{AG2008,PR2013,D79,golseparticle,villani}.
Furthermore, the sublinearity condition (v3) guarantees long-time existence of the solution.
Concerning the operator $\cT$, it must comply with analogous local Lipschitz continuity conditions and sublinearity conditions, see (T1)-(T2); in addition, we require that 
\begin{itemize}
\item constants belong to the kernel of $\cT(x,\Psi)$, namely $\cT(x,\Psi)1=0$;
\item there exists $\delta>0$ such that $\cT(x,\Psi)+\delta I\geq0$.
\end{itemize}
The two conditions above (see (T0) and (T3) below for a precise statement) ensure that the velocities $\dot\lambda_i$'s lie on the tangent plane to the probabilities on $\{F,L\}$ and that the positivity of the measures are preserved, respectively.
Altogether, conditions (v1)-(v3) and (T0)-(T3) allow us to show that the particle model, which we can rewrite in the form
\begin{equation}\label{I4}
\dot y_i=\vec{\dot x_i}{\dot\lambda_i}=\vec{v_{\Psi^N}(x_i,\lambda_i)}{\cT^*(x_i,\Psi^N)\lambda_i}\eqqcolon b_{\Psi^N}(y_i),
\end{equation}
is well-posed (see Proposition~\ref{T133}).

We subsequently show that in the limit as $N\to\infty$ the empirical measures $\Psi^N$ converge to a continuous path of probability measures $\Psi_t$  that solves the continuity-type equation
\begin{equation}\label{I5}
\partial_t\Psi_t+\div(b_{\Psi_t}\Psi_t)=0
\end{equation}
in the space of probability measures on the product space $Y=\R{d}\times\cP(U)$.

A key point of the proof is that the solutions to \eqref{I5}, which are defined via duality with test functions (Eulerian solutions, see Definition~\ref{T038}), can be equivalently characterized as Lagrangian solution satisfying the fixed-point equation
\begin{equation}\label{I6}
\Psi_t=\bY_\Psi(t,0,\cdot)_\# \overline\Psi \qquad\text{for every $t\geq0$,}
\end{equation}
where $\bY(t,0,\cdot)$ is the transition map (see Definition~\ref{T142}) associated with the ODE \eqref{I4} (with $\Psi$ in place of $\Psi^N$), $\overline\Psi$ is a given initial distribution of agents and probability on the labels, and the symbol $\#$ denotes the push-forward measure (see Definition~\ref{pushforward}).

In Section~\ref{sec:LF} we also retrieve and extend the main result of \cite{ABRS2018} as a special case of our analysis. 
We associate with a solution to \eqref{I5} the followers and leaders distributions 
\begin{equation}\label{I7}
\mu^F_\Psi(B):=\int_{B \times \cP(\{F, L\})} \lambda(\{F\}) \,\mathrm{d}\Psi(x, \lambda),\qquad \mu^L_\Psi(B):=\int_{B \times \cP(\{F, L\})} \lambda(\{L\})) \,\mathrm{d}\Psi(x, \lambda),
\end{equation}
for each Borel set $B \subset \mathbb{R}^d$. 
If the vector field $v_\Psi$ has the special structure \eqref{I1} and under suitable structural assumptions on the transition rates $\alpha_F$ and $\alpha_L$ (see \eqref{T177} below), using the definition of Eulerian solution, we are able to show that $\mu^F_\Psi$ and $\mu^L_\Psi$ are indeed the unique solutions to the system
\begin{align}\label{I8}
\left\{\begin{aligned}
\partial_t \mu^F_t& = -\mathrm{div}\big((K^{F}\star\mu^F_t + K^{L}\star\mu^L_t)\mu^F_t\big) - \alpha_F(x,\mu^F_t,\mu^L_t)\mu^F_t + \alpha_L(x, \mu^F_t,\mu^L_t)\mu^L_t,\\
\partial_t \mu^L_t& = -\mathrm{div}\big((K^{F}\star\mu^F_t + K^{L}\star\mu^L_t)\mu^L_t\big) +\alpha_F(x, \mu^F_t,\mu^L_t)\mu^F_t -\alpha_L(x, \mu^F_t,\mu^L_t)\mu^L_t
\end{aligned}\right.
\end{align}
considered in \cite{ABRS2018}.  Notice that, as discussed in Remark \ref{Tfin}, rewriting \eqref{T140} as a system of equation in $\mu^F$ and $\mu^L$ is not possible for a velocity field explicitly depending on $\lambda$ such as
\[
\begin{split}
v_{\Psi^N}(x_i, \lambda_i)=& \lambda_i(\{F\})\left(\frac{1}{N}\sum_{j=1}^N K^{FF}(x_i-x_j)\lambda_j(\{F\})+\frac{1}{N}\sum_{j=1}^N K^{LF}(x_i-x_j)\lambda_j(\{L\})\right)+\\
&\lambda_i(\{L\})\left(\frac{1}{N}\sum_{j=1}^N K^{FL}(x_i-x_j)\lambda_j(\{F\})+\frac{1}{N}\sum_{j=1}^N K^{LL}(x_i-x_j)\lambda_j(\{L\})\right)
\end{split}
\]
considered for instance in \cite{francesco2013measure}. In the presence of exchange rates, the correct mean-field description of the above particle system is given by \eqref{T140} in the product space $\mathbb{R}^d\times \cP_1(\{F, L\})$.

The paper is organized as follows. 
In Section~\ref{sec:prel}, we recall the basic notions of measure theory that will be needed in the sequel, and prove a corollary of a theorem by Brezis \cite[Sect.~I.3, Thm.~1.4, Cor.~1.1]{Brezis} on the well-posedness of ODE's in Banach spaces.
In Section~\ref{sec:abstract_mod} we introduce our abstract model and we prove our main result, Theorem~\ref{T150}, on the mean-field limit of the dynamics.
Section~\ref{sec:LF} is devoted to the special case of $U=\{F,L\}$, modeling the leader-follower dynamics. We apply the abstract results to this case and recover the results of \cite{ABRS2018}, extending them to the case of $x$-dependent transition rates.
Finally, Section~\ref{discreteU} we extend the results of Section~\ref{sec:LF} to any finite numbers of labels, whereas in Section~\ref{continuousU} we discuss some explicit examples of velocity fields $v_\Psi$ and transition operators $\cT(x,\Psi)$ which are encompassed by our setting in the case of a continuum of labels, and we compare them with those considered in \cite{AFMS2018}.

\section{Preliminaries} \label{sec:prel}
\subsection{Basic notation} 
If $(X,\sfd_X)$ is a metric space, we denote by $\mathcal{M}(X)$ the space of signed Borel measures in $X$ with finite total variation,
by $\mathcal{M}_+(X)$ and $\cP(X)$ the convex subsets of nonnegative measures and probability measures, respectively. 
The notation $\cP_c(X)$ will be used for measures having compact support in $X$. 
For $\mu \in\mathcal{M}(X)$, $|\mu|$ denotes the total variation of $\mu$.  
If we denote by  $C_0(X)$ the space of continuous functions vanishing at the boundary of $X$, and by $C_b(X)$ the space of bounded continuous functions, the weak$^*$ and narrow convergence in $\mathcal{M}(X)$ are defined by the convergence of the duality products
\[
\int_X \phi\,\mathrm{d}\mu_h \to \int_X \phi\,\mathrm{d}\mu\,,\qquad h \to \infty
\]
for each $\phi \in C_0(X)$ and $\phi \in C_b(X)$, respectively.

Whenever $X=\mathbb{R}^d$, $d\ge 1$, it remains understood that it is endowed with the Euclidean norm (and induced distance), which shall be simply denoted by $|\cdot|$. 

For a Lipschitz function $f\colon X\to\mathbb{R}$ we denote by
$$
\mathrm{Lip}(f)\coloneqq\sup_{x,y\in X \atop x\neq y}\frac{|f(x)-f(y)|}{\sfd_X(x,y)}
$$
the Lipschitz constant. The notations $\mathrm{Lip}(X)$ and $\mathrm{Lip}_b(X)$ will be used for the spaces of Lipschitz and bounded Lipschitz functions on $X$, respectively. Both are normed spaces with the norm $\lVert f\rVert \coloneqq \lVert f\rVert_\infty+ \mathrm{Lip}(f)$.

In a complete and separable metric space $(X,\sfd_X)$,  we shall use the Kantorovich-Rubinstein
 distance $\mathcal{W}_1
 $ in the class $\cP(X)$, defined as
$$
\mathcal{W}_1(\mu,\nu)\coloneqq\sup\bigg\{\int_X\varphi\,\mathrm{d}\mu-\int_X\varphi\,\mathrm{d}\nu: \varphi\in\mathrm{Lip}_b(X), \,\mathrm{Lip}(\varphi)\leq 1\bigg\}
$$
or equivalently (thanks to the Kantorovich duality) as
$$
\mathcal{W}_1(\mu,\nu)\coloneqq\inf\bigg\{\int_{X\times X} {\crd  \sfd_X(x,y)}\,\mathrm{d}\Pi(x,y): 
\Pi(A\times X)=\mu(A),\,\,\Pi(X\times B)=\nu(B)\bigg\},
$$
involving couplings $\Pi$ of $\mu$ and $\nu$. Notice that $\cW_1(\mu,\nu)$ is finite
if $\mu$ and $\nu$ belong to the space
\begin{equation*}
\cP_1(X)\coloneqq \bigg\{\mu\in\cP(X):  \text{$\int_X \sfd_X(x,\bar x)\,\mathrm{d}\mu(x)<+\infty$ for some $\bar x\in X$}\bigg\}
\end{equation*}
and that $(\cP_1(X),\mathcal{W}_1)$ is complete if $(X,\sfd_X)$ is complete.
For a positive measure $\mu \in\mathcal{M}_+(E)$, for $E$ being a Banach space, we define the first moment $m_1(\mu)$ as
\[
m_1(\mu)\coloneqq\int_{E} \lVert x \rVert_E\,\mathrm{d}\mu\,.
\]
Notice that, for a probability measure $\mu$, finiteness of the integral above is equivalent to $\mu \in \cP_1(E)$, whenever $E$ is endowed with the distance induced by the norm $\lVert\cdot  \rVert$.

We now recall the definition of push-forward measure.
\begin{defin}\label{pushforward}
Let  $\mu\in\mathcal{M}_+(X)$ and  $f  \colon X\to Z$ a
$\mu$-measurable function be given.   The push-forward measure $f_\#\mu\in\mathcal{M}_+(Z)$ is defined by $f_\#\mu(B)=\mu(f^{-1}(B))$ for any Borel set $B\subset Z$.
The push-forward measures has the same total mass as $\mu$, namely $\mu(X)=f_\#\mu(Z)$.
It also holds the change of variables formula
$$
\int_Z g\,\mathrm{d}f_\#\mu=\int_X g\circ f\,\mathrm{d}\mu
$$
whenever either one of the integrals makes sense.
\end{defin}

For $E$ being  a Banach space, the notation $C^1_b(E)$ will be used to denote the subspace of $C_b(E)$ of functions having bounded continuous Fr\'echet differential at each point. The notation $D\phi(\cdot)$ will be used to denote the Fr\'echet differential. In the case of a function $\phi \colon [0,T]\times E \to \mathbb{R}$, the simbol $\partial_t$ will be used to denote partial differentiation with respect to $t$, while $D$ will only stand for differentiation with respect to the variables in $E$.

\subsection{Functional setting} The space of labels $U$ will be assumed  to be a compact metric space.
The state of the system is described by $y\coloneqq (x,\lambda)\in \R{d}\times\cP(U)\eqqcolon Y$. 
The component $x\in\R{d}$ describes the location of an agent in space, whereas the component $\lambda\in\cP(U)$ describes the distribution of labels of the agent.
A probability distribution $\Psi\in\cP(Y)$ denotes a distribution of agents with labels.

To define the functional setting for the dynamics, we need to consider the \emph{free space}
\begin{equation}\label{P1}
\cF(U)\coloneqq \overline{\spann(\cP(U))}^{\lVert\cdot\rVert_{\BL}}\subset (\Lip(U))'\,,
\end{equation}
which is also called in the literature Arens-Eells space, see \cite{AP,AE1956} and \cite[Chapter~3]{Weaver2018}.
The closure in \eqref{P1} is taken with respect to the \emph{bounded Lipschitz} norm $\lVert\cdot\rVert_{\BL}$, which is defined, for $\mu\in(\Lip(U))'$, by
\begin{equation*}
\lVert \mu \rVert_{\BL}\coloneqq\sup \big\{\langle \mu,\varphi\rangle:  \varphi\in \Lip(U),  \|\varphi\|_{\Lip}\leq 1\big\}.
\end{equation*}
With the free space $\cF(U)$ at hand, we define $\overline Y\coloneqq \R{d}\times\cF(U)$ and the norm $\lVert\cdot\rVert_{\overline Y}$ by
\begin{equation}\label{P2}
\lVert y\rVert_{\overline Y}=\lVert(x,\lambda)\rVert_{\overline Y}\coloneqq \lvert x \rvert+\lVert \lambda \rVert_{\BL}\,;
\end{equation}
and we pose $\lVert y\rVert_Y=\lVert y\rVert_{\overline Y}$.

For a given $R>0$, we denote by $B_R$ the closed ball of radius $R$ in $\R{d}$ and by $B_R^Y$ the ball of radius $R$ in $Y$, namely $B_R^Y=\{y\in Y:\lVert y\rVert_{\overline Y}\leq R\}$, and observe that it is a compact set, since $Y$ is locally compact by our assumptions on $U$.
The Banach space structure of $\overline Y\supset Y$ allows us to define the first moment $m_1(\Psi)$ for a probability measure $\Psi \in \cP(Y)$ as
\[
m_1(\Psi)\coloneqq\int_Y \lVert y\rVert_{\overline Y} \,\mathrm{d}\Psi
\]
so that the space $\cP_1(Y)$ can be equivalently characterized as
$$\cP_1(Y)=\{\Psi\in\cP(Y): m_1(\Psi)<+\infty\}.$$
We will sometimes use the notation $\cP(K)$ to denote probability measures with support contained on a given compact subset $K\subset Y$. Notice that trivially we have $\cP(K)\subset \cP_1(Y)$.

\subsection{Well-posedness of ODE's in Banach spaces} We recall here a theorem by Brezis \cite[Sect.~I.3, Thm.~1.4, Cor.~1.1]{Brezis} on the well-posedness of ODE's in Banach spaces. 
\begin{theorem}\label{thm:Brezis}
Let $(E,\|\cdot\|_E)$ be a Banach space, $C$ a closed convex subset of $E$, and let $A(t,\cdot)\colon C\to E$, $t\in [0,T]$, be a family of operators satisfying the following properties: 
\begin{enumerate}
\item there exists a constant $L\geq 0$ such that for every $c_1,\,c_2\in C$ and $t\in [0,T]$
\begin{equation}\label{eq:15}
\|A(t,c_1)-A(t,c_2)\|_E\le L\|c_1-c_2\|_E;
\end{equation}
\item for every $c\in C$ the map $t\mapsto A(t,c)$ is continuous in $[0,T]$;
\item for every $R>0$ there exists $\theta>0$ such that
\begin{equation}\label{eq:16}
c\in C,\ \|c\|_E\le R\quad \Rightarrow\quad c+\theta A(t,c)\in C.
\end{equation}
\end{enumerate}
Then for every $\bar c\in C$ there exists a unique curve $c\colon[0,T]\to C$ of class $C^1$ satisfying $c_t\in C$ for all $t\in [0,T]$ and
\begin{equation}\label{eq:17}
\frac{\de}{\de t}c_t=A(t,c_t)\quad\text{in }[0,T],\qquad  c_0=\bar c.
\end{equation}
Moreover, if $c^1,\,c^2$ are the solutions starting from the initial data $\bar c^1,\,\bar c^2\in C$ respectively, we have
\begin{equation}\label{eq:18}
\|c^1_t-c^2_t\|_E\le \mathrm e^{Lt}\|\bar c^1-\bar c^2\|_E,\qquad\text{for every $t\in [0,T]$.}
\end{equation}
\end{theorem}
For our purposes, we need the following generalization.
\begin{cor}\label{cor:Brezis}
Let hypotheses (ii) and (iii) of Theorem~\ref{thm:Brezis} hold for a family of operators $A(t,\cdot)\colon C\to E$, $t\in[0,T]$.
Assume, in addition that
\begin{enumerate}
\item[(i')] for every $R>0$ there exists a constant $L_R\geq 0$ such that for every $c_1,\,c_2\in C\cap B_R$ and $t\in [0,T]$
\begin{equation}\label{eq:151}
\|A(t,c_1)-A(t,c_2)\|_E\le L_R\|c_1-c_2\|_E;
\end{equation}
\item[(i'')] there exists $M>0$ such that for every $c\in C$, there holds 
\begin{equation}\label{eq:152}
\lVert A(t,c)\rVert_E\leq M(1+\lVert c\rVert_E).
\end{equation}
\end{enumerate}
Then for every $\bar c\in C$ there exists a unique curve $c\colon[0,T]\to C$ of class $C^1$ satisfying $c_t\in C$ for all $t\in [0,T]$ and
\begin{equation}\label{eq:171}
\frac{\de}{\de t}c_t=A(t,c_t)\quad\text{in }[0,T],\qquad  c_0=\bar c.
\end{equation}
Moreover, if $c^1,\,c^2$ are the solutions starting from the initial data $\bar c^1,\,\bar c^2\in C\cap B_R$ respectively, there exists a constant $L=L(M,R,T)>0$ such that
\begin{equation}\label{eq:181}
\|c^1_t-c^2_t\|_E\le \mathrm e^{Lt}\|\bar c^1-\bar c^2\|_E,\qquad\text{for every $t\in [0,T]$.}
\end{equation}
\end{cor}
\begin{proof}
Let us fix the initial datum $\bar c\in C$, and let us choose $\bar R\coloneqq (\lVert \bar c\rVert_E+MT)e^{MT}$.
Consider a smooth function with compact support $\chi\colon \R+\to [0,1]$ such that $\chi(r)=1$ for every $r\leq \bar R$ and set $B(t,c)\coloneqq \chi(\lVert c\rVert_E)A(t,c)$. 
Then one can see that $B$ satisfies hypotheses (i) and (ii) of Theorem~\ref{thm:Brezis}.
To see that hypothesis (iii) is also satisfied, it suffices to notice that, by convexity and since $0\leq\chi\leq 1$, $c+\theta \chi(\lVert c\rVert_E)A(t,c)\in C$ whenever $c+\theta A(t,c)\in C$.
Therefore there exists a unique solution $t\mapsto c(t)$ of class $C^1$ of 
\begin{equation}\label{eq:182}
\frac{\de}{\de t} c_t=B(t,c_t)\quad\text{in }[0,T],\qquad  c_0=\bar c.
\end{equation}
Using again that $0\leq \chi\leq 1$ and \eqref{eq:152}, one can see that 
\begin{equation}\label{eq:183}
\lVert c_t\rVert_E\leq \lVert \bar c\rVert_E+MT+M\int_0^T \lVert c_s\rVert_E\,\de s,
\end{equation}
hence Gronwall's Lemma implies that $\lVert c_t\rVert_E\leq \bar R$ for every $t\in[0,T]$.
With this, $c_t$ solves \eqref{eq:171}. A similar argument shows that any other solution $t\mapsto \hat c_t$ to \eqref{eq:171} must satisfy $\lVert \hat c_t\rVert_E\leq \bar R$ for every $t\in[0,T]$. Thus, uniqueness of solutions for \eqref{eq:171} follows from the uniqueness of solutions to \eqref{eq:182}.
A similar argument also yields \eqref{eq:181}.
\end{proof}

\section{The abstract model} \label{sec:abstract_mod}
The state of our system is described pairs $y\coloneqq(x,\lambda)$. The element $x\in\R{d}$ denotes the position of the agents, whereas the element $\lambda\in\cP(U)$ denotes a (probability) distribution over the space 
$U$, which we assume to be a compact metric space, which can be interpreted as a space of strategies (as in \cite{AFMS2018}, in which case an element of $\cP(U)$ is a mixed strategy) or as a space of labels (as in \cite{ABRS2018}, where the case $U=\{\text{leader},\text{follower}\}$ was considered).
A distribution of states 
will be described by an element $\Lambda\in\cP(Y)$, where $Y\coloneqq \R{d}\times\cP(U)$. 
We will be concerned with the evolution of $\Lambda$, given an initial $\Lambda^0$, determined by the laws of evolution of $x$ and $\lambda$, which are going to be discussed below.

For $y=(x,\lambda)\in  Y$ and $\Psi\in\cP_1(Y)$, we define a vector field $b_\Psi\colon  Y\to\overline Y$ through
\begin{equation}\label{T034}
b_\Psi(y)\coloneqq \vec{v_\Psi(y)}{\cT^*(x,\Psi)\lambda}
\end{equation}
The first component of $b_\Psi$ is a velocity field in $\R{d}$ determined by the global state of the system $\Psi$; the second component is expressed in terms of the adjoint $\cT^*(x,\Psi)$ of an operator $\cT(x,\Psi)$ which sees the location of the agents and the global state of the system around them.
In order to state the regularity assumptions that we make on $b_\Psi$, we will discuss separately the assumptions on $v_\Psi$ and on $\cT$.

We assume that the velocity field $v_\Psi\colon  Y\to\R{d}$ satisfies the following conditions:
\begin{itemize}
\item[(v1)] for every $R>0$, for every $\Psi\in\cP(B_R^Y)$, $v_\Psi\in\Lip(B_R^{Y};\R{d})$ uniformly with respect to $\Psi$, namely there exists a constant $L_{v,R}>0$ such that
\begin{equation}\label{T100}
\lvert v_\Psi(y^1)-v_\Psi(y^2)\rvert \leq L_{v,R}\lVert y^1-y^2 \rVert_{\overline Y}\,;
\end{equation}
\item[(v2)] for every $R>0$, for every $\Psi\in\cP(B_R^Y)$, there exists a constant $L_{v,R}>0$ such that for every $y\in B_R^{Y}$, and for every $\Psi^1,\Psi^2\in\cP(B_R^Y)$
\begin{equation}\label{T101}
\lvert v_{\Psi^1}(y)-v_{\Psi^2}(y)\rvert \leq L_{v,R}\cW_1(\Psi^1,\Psi^2);
\end{equation}
\item[(v3)] there exists $M_v>0$ such that for every $y\in Y$ and for every $\Psi\in \cP_1(Y)$ there holds
\begin{equation}\label{T102}
\lvert v_\Psi(y)\rvert \leq M_v \big(1+\lVert y \rVert_{\overline Y} + m_1(\Psi)\big). 
\end{equation}
\end{itemize}

We now describe the assumptions on $\cT$. 
For $(x,\Psi)\in\R{d}\times\cP_1(Y)$, let $\cT(x,\Psi)\colon \Lip(U)\to \Lip(U)$ be an operator such that 
\begin{itemize}
\item[(T0)] for every $(x,\Psi)\in\R{d}\times\cP_1(Y)$, constants are in the kernel of $\cT(x,\Psi)$, that is
\begin{equation}\label{T122}
\cT(x,\Psi)1=0;
\end{equation}
\item[(T1)] for every $(x,\Psi)\in \R{d}\times\cP_1(Y)$, there exists a constant $M_\cT>0$ such that the operator norm satisfies
\begin{equation}\label{T111}
\lVert \cT(x,\Psi)\rVert_{\cL(\Lip(U);\Lip(U))}\leq M_\cT\big(1+\lvert x\rvert+m_1(\Psi)\big);
\end{equation}
\item[(T2)] for every $R>0$ there exists $L_{\cT,R}>0$ such that, for every $(x^1,\Psi^1),(x^2,\Psi^2)\in B_R\times\cP(B_R^Y)$, 
\begin{equation}\label{T007}
\lVert \cT(x^1,\Psi^1)-\cT(x^2,\Psi^2)\rVert_{\cL(\Lip(U);\Lip(U))}\leq L_{\cT,R}\big(|x^1-x^2|+\cW_1(\Psi^1,\Psi^2)\big);
\end{equation}
\item[(T3)] for every $R>0$ there exists $\delta_R>0$ such that for every $(x,\Psi)\in B_R\times\cP_1(Y)$ we have
\begin{equation}\label{T118}
\cT(x,\Psi)+\delta_R I\geq0.
\end{equation}
\end{itemize}
The following lemma is easily proved.
\begin{lemma}\label{T020}
Let $\cT(x,\Psi)\colon \Lip(U)\to \Lip(U)$ satisfy (T0)-(T3) above and let $\cT^*(x,\Psi)\colon\cF(U)\to\cF(U)$ be its adjoint. 
Then
\begin{itemize}
\item[(T*1)] for every $(x,\Psi)\in \R{d}\times\cP_1(Y)$, there exists $M_\cT>0$ such that the operator norm satisfies
\begin{equation}\label{T120}
\lVert \cT^*(x,\Psi)\rVert_{\cL(\cF(U);\cF(U))}\leq M_\cT(1+\lvert x\rvert+m_1(\Psi));
\end{equation}
\item[(T*2)] for every $R>0$ there exists $L_{\cT,R}>0$ such that, for every $(x^1,\Psi^1),(x^2,\Psi^2)\in B_R\times\cP(B_R^Y)$, 
\begin{equation}\label{T009}
\lVert \cT^*(x^1,\Psi^1)-\cT^*(x^2,\Psi^2)\rVert_{\cL(\cF(U);\cF(U))}\leq L_{\cT,R}\big(|x^1-x^2|+\cW_1(\Psi^1,\Psi^2)\big),
\end{equation}
\item[(T*3)] for every $R>0$ there exists $\delta_R>0$ such that 
\begin{equation}\label{T119}
\cT^*(x,\Psi)+\delta_R I\geq 0,
\end{equation}
\end{itemize}
\end{lemma}
\begin{proof}
To see that (T*2) holds, we apply the definition of operator norm
\begin{equation}\label{T107}
\lVert \cT^*(x,\Psi)\rVert_{\cL(\cF(U);\cF(U))}\coloneqq \sup\big\{ \lVert \cT^*(x,\Psi)\mu\rVert_{\BL}: \mu\in\cF(U),\lVert\mu\rVert_{\BL}\leq1\big\}.
\end{equation}
For $(x^1,\Psi^1),(x^2,\Psi^2)\in B_R\times\cP(B_R^Y)$, we have,
\begin{equation*}
\begin{split}
\lVert \cT^* & (x^1,\Psi^1)- \cT^*(x^2,\Psi^2)\rVert_{\cL(\cF(U);\cF(U))}\\
=& \sup\big\{ \lVert (\cT^*(x^1,\Psi^1)-\cT^*(x^2,\Psi^2))\mu\rVert_{\BL}: \mu\in\cF(U),\lVert\mu\rVert_{\BL}\leq1\big\} \\
=& \sup\Big\{ \sup\big\{ \langle (\cT^*(x^1,\Psi^1)-\cT^*(x^2,\Psi^2))\mu, \varphi\rangle: \lVert\varphi\rVert_{\Lip}\leq1\big\} : \mu\in\cF(U),\lVert\mu\rVert_{\BL}\leq1\Big\} \\
=& \sup\Big\{ \sup\big\{ \langle \mu, (\cT(x^1,\Psi^1)-\cT(x^2,\Psi^2))\varphi\rangle:  \lVert\varphi\rVert_{\Lip}\leq1\big\} : \mu\in\cF(U),\lVert\mu\rVert_{\BL}\leq1\Big\} \\
\leq& \sup\Big\{ \sup\big\{ \lVert\mu\rVert_{\BL} \lVert(\cT(x^1,\Psi^1)-\cT(x^2,\Psi^2))\varphi\rVert_{\Lip}:  \lVert\varphi\rVert_{\Lip}\leq1\big\} : \mu\in\cF(U),\lVert\mu\rVert_{\BL}\leq1\Big\} \\
\leq & \sup\big\{ \lVert(\cT(x^1,\Psi^1)-\cT(x^2,\Psi^2))\varphi\rVert_{\Lip}:  \lVert\varphi\rVert_{\Lip}\leq1\big\} \\
\leq & \lVert \cT(x^1,\Psi^1)-\cT(x^2,\Psi^2)\rVert_{\cL(\Lip(U);\Lip(U))}\leq L_{\cT,R}\big(|x^1-x^2|+\cW_1(\Psi^1,\Psi^2)\big),
\end{split}
\end{equation*}
and (T*2) follows from (T2). A similar argument using \eqref{T107} also gives (T*1).
To prove (T*3), ler $R>0$ be fixed and let $\delta_R>0$ be such that \eqref{T118} holds.
Then 
$$
\cT^*(x,\Psi)+\delta_R I=(\cT(x,\Psi)+\delta_R I)^*\geq 0,
$$
since the adjoint operator preserves positivity.
\end{proof}
%
%
\begin{proposition}\label{T105}
For $y\in Y$ and $\Psi\in\cP_1(Y)$, define $b_\Psi(y)$ as in \eqref{T034}.
Assume that $v_\Psi\colon Y\to\R{d}$ satisfies (v1)-(v3) and $\cT(x,\Psi)\colon \Lip(U)\to\Lip(U)$ satisfies (T0)-(T3). Then
\begin{itemize}
\item[(i)] for every $R>0$, for every $\Psi\in \cP(B_R^Y)$, and for every $y^1,y^2\in B_R^Y$, there exists $L_R>0$ such that 
\begin{equation}\label{T106}
\lVert b_\Psi(y^1)-b_\Psi(y^2)\rVert_{\overline Y}\leq L_R \lVert y^1-y^2\rVert_{\overline Y}\,;
\end{equation}
\item[(ii)] for every $R>0$, for every $\Psi^1,\Psi^2\in \cP(B_R^Y)$, and for every $y\in B_R^Y$, there exists $L_R>0$ such that 
\begin{equation}\label{T117}
\lVert b_{\Psi^1}(y)-b_{\Psi^2}(y)\rVert_{\overline Y}\leq L_R \cW_1(\Psi^1,\Psi^2);
\end{equation}
\item[(iii)] for every $R>0$, there exists $\theta>0$ such that for every $y\in B_R^Y$ and for every $\Psi\in \cP(B_R^Y)$
\begin{equation}\label{T109}
y+\theta b_\Psi(y)\in Y;
\end{equation}
\item[(iv)] there exists $M>0$ such that for every $y\in Y$ and for every $\Psi\in \cP_1(Y)$ there holds
\begin{equation}\label{T110}
\lVert b_{\Psi}(y)\rVert_{\overline Y}\leq M \big(1+ \lVert y\rVert_{\overline Y} +m_1(\Psi)\big)\,.
\end{equation}
\end{itemize}
\end{proposition}
\begin{proof}
Property (i) is a consequence of \eqref{T100} and \eqref{T009}; property (ii) follows from \eqref{T101} and \eqref{T009}.
Observing that for all $y=(x,\lambda)\in Y$ we have that $\lVert \lambda\rVert_{\BL}\leq 1$, we obtain \eqref{T110} from \eqref{T102} and \eqref{T120}.
Finally, to prove (iii), since $\R{d}$ is convex, we simply have to show that for every $R>0$, there exists $\theta>0$ such that for every $y=(x,\lambda)\in B_R^Y$ and for every $\Psi\in \cP(B_R^Y)$
\begin{equation}\label{T121}
\lambda+\theta \cT^*(x,\Psi)\lambda\in \cP(U).
\end{equation}
From \eqref{T119}, we have that, for $\theta=1/\delta_R$, $\lambda+\theta \cT^*(x,\Psi)\lambda$ is a positive measure.
Since by \eqref{T122} $\langle \cT^*(x,\Psi)\lambda, 1\rangle=\langle \lambda, \cT(x,\Psi)1\rangle=0$, we get \eqref{T121} and we conclude the proof.
 \end{proof}

\subsection{The discrete problem and statement of the main result}\label{section:discrete}
We consider a particle system of $N$ agents evolving according to
\begin{equation}\label{T130}
\begin{cases}
\dot x_t^i = v_{\Lambda_t^N}(x_t^i,\lambda_t^i), \\
\dot \lambda_t^i = \cT^*(x_t^i,\Lambda_t^N)\lambda_t^i,
\end{cases}
\qquad\text{for $i=1,\ldots,N$, $t\in[0,T]$,}
\end{equation}
where $x^i\in\R{d}$, $\lambda^i\in\cP(U)$ for each $i\in\{1,\ldots,N\}$, and 
\begin{equation}\label{T131}
\Lambda_t^N\coloneqq \frac1N\sum_{i=1}^N \delta_{(x_t^i,\lambda_t^i)}
\end{equation}
is the empirical measure associated with the system.
Recalling the definition of $b$ in \eqref{T034}, the evolution \eqref{T130} can be written in compact form as
\begin{equation}\label{T132}
y_t^i=b_{\Lambda_t^N}(y_t^i), \qquad\text{for $i=1,\ldots,N$, $t\in[0,T]$.}
\end{equation}
We first discuss the well posedness of system \eqref{T130} for every choice of an initial datum $\bar y^i=(\bar x^i,\bar \lambda^i)$, for $i=1,\ldots,N$.
\begin{proposition}\label{T133}
Assume that for every $y\in Y$ and $\Psi\in\cP_1(Y)$ the velocity $v_\Psi\colon Y\to\R{d}$ satisfies (v1)-(v3) and the operator $\cT(x,\Psi)\colon \Lip(U)\to\Lip(U)$ satisfies (T0)-(T3). 
Then, for every choice of $\bar y^i\in Y$, $i=1,\ldots,N$, the system \eqref{T132} has a unique solution.
\end{proposition}
\begin{proof}
We introduce the vector-valued variable $\by\coloneqq(y^1,\ldots,y^N)\in Y^N\subset\overline Y^N$, which we endow with the norm
\begin{equation}\label{T136}
\lVert \by \rVert_{\overline Y^N}\coloneqq \frac1N\sum_{i=1}^N \lVert y^i\rVert_{\overline Y},
\end{equation}
and the associated empirical measure $\Lambda^N\coloneqq\frac1N \sum_{i=1}^N \delta_{y^i}$, which belongs to $\cP(B_R^Y)$ whenever $\by\in (B_R^Y)^N$.
Consider the map $\bb^N\colon Y^N\to\overline Y^N$ whose components are defined through 
\begin{equation}\label{T134}
b_i^N(\by)\coloneqq b_{\Lambda^N}(y^i).
\end{equation}
Then the Cauchy problem associated with \eqref{T132} can be written as
\begin{equation}\label{T135}
\begin{cases}
\dot \by_t=\bb^N(\by_t), \\
\by_0=\bar \by.
\end{cases}
\end{equation}
In order to apply Corollary~\ref{cor:Brezis} to the system above, we first notice that assumption (ii) is automatically satisfied since the system is autonomous. To check the other assumptions, we fix a ball $B_R^{Y^N}$ and notice that $B_R^{Y^N}\subset (B_{R}^Y)^N$.

Applying \eqref{T109} with $\Psi=\Lambda^N$ to each component $y^i$ of $\by$, we get that assumption (iii) of Corollary~\ref{cor:Brezis} is satisfied.

We now show that assumption (i') holds. Fix $\by_1,\by_2\in B_R^{Y^N}\subset (B_R^Y)^N$, and let $\Lambda_1^N,\Lambda_2^N$ be the associated empirical measures. 
Notice that 
\begin{equation}\label{T138}
\cW_1(\Lambda_1^N,\Lambda_2^N) \leq \frac1N\sum_{i=1}^N \lVert y_1^i - y_2^i \rVert_{\overline Y} = \lVert \by_1 - \by_2 \rVert_{\overline Y^N}.
\end{equation}
With this, by the triangle inequality, \eqref{T106}, and \eqref{T117}, we estimate
\begin{equation}\label{T137}
\begin{split}
\lVert \bb^N(\by_1)-\bb^N(\by_2)\rVert_{\overline Y^N} = & \frac1N \sum_{i=1}^N \lVert b_{\Lambda_1^N}(y_1^i) - b_{\Lambda_2^N}(y_2^i) \rVert_{\overline Y} \\
\leq & L_R \cW_1(\Lambda_1^N,\Lambda_2^N) + \frac{L_R}N\sum_{i=1}^N \lVert y_1^i - y_2^i \rVert_{\overline Y} \leq 2L_R \lVert \by_1 - \by_2 \rVert_{\overline Y^N},
\end{split}
\end{equation}
which yields \eqref{eq:151}. To see that also \eqref{eq:152}, and therefore assumption (i''), holds, we apply \eqref{T110} and obtain, upon noticing that $m_1(\Lambda^N)= \lVert \by\rVert_{\overline Y^N}$,
\begin{equation}\label{T139}
\lVert \bb^N(\by)\rVert_{\overline Y^N}=\frac1N \sum_{i=1}^N \lVert b_{\Lambda^N}(y^i)\rVert_{\overline Y} \leq \frac{M}{N} \sum_{i=1}^N \big(1+\lVert y^i\rVert_{\overline Y}+m_1(\Lambda^N)\big)= M\big(1+ 2\lVert \by\rVert_{\overline Y^N}\big),
\end{equation}
Therefore we can apply Corollary~\ref{cor:Brezis}, which proves the statement.
\end{proof}
We are now in a position to state the main result of our paper, concerning the mean-field limit as $N\to\infty$ of the solutions $(y_t^1,\ldots,y_t^N)$ to \eqref{T132}, or equivalently the limiting behavior of the associated empirical measures $\Lambda_t^N$.
In order to do so, we first need to recall the concept of Eulerian solution to the continuity equation.
\begin{defin}[Eulerian solution]\label{T038}
Let $\Lambda\in C^0([0,T];(\cP_1(Y),\cW_1))$ and let $\bar\Lambda\in\cP_c(Y)$ be a given initial datum.
We say that $\Lambda$ is a Eulerian solution to the initial value problem for the equation
\begin{equation}\label{T140}
\partial_t \Lambda_t+\div(b_{\Lambda_t}\Lambda_t)=0
\end{equation}
starting from $\bar\Lambda$ if and only if $\Lambda_0=\bar\Lambda$ and, for every $\phi\in C^1_b([0,T]\times \overline Y)$,
\begin{equation}\label{T037}
\int_Y \phi(t,y)\,\de\Lambda_t(y)- 
\int_Y \phi(0,y)\,\de\Lambda_0(y) 
=
\int_0^t\int_Y \big(\partial_t\phi(s,y)+ D\phi(s,y)\cdot b_{\Lambda_s}(y)\big)\,\de\Lambda_s(y)\de s,
\end{equation}
where $D\phi(s,y)$ is the Fr\'echet differential of $\phi$ in the $y$ variable. 
\end{defin}
The main result of our paper is the following theorem, stating the existence of a unique Eulerian solution to \eqref{T140} and its characterization as the mean-field limit of solutions to the discrete problem \eqref{T132}.
\begin{theorem}\label{T150}
Let  $r>0$ and $\bar\Lambda\in\cP(B_r^Y)$ be a given initial datum.
Then 
\begin{itemize}
\item[(i)] there exists a unique Eulerian solution $t\mapsto\Lambda_t$ to \eqref{T140} starting from $\bar\Lambda$;
\item[(ii)] if $\bar\Lambda^N=\frac1N\sum_{i=1}^N \delta_{\bar y^{N,i}}$ is a sequence of atomic measures in $\cP(B^Y_r)$ such that
$$\lim_{N\to\infty} \cW_1(\bar \Lambda^N,\bar\Lambda)=0$$
and, for fixed $N$, $\Lambda_t^N$ are the empirical measures associated with the unique solution to \eqref{T132} with initial datum $\bar y^{N,i}$, we have
$$\lim_{N\to\infty} \cW_1(\Lambda_t^N,\Lambda_t)=0\qquad\text{uniformly with respect to $t\in[0,T]$.}$$
\end{itemize}
\end{theorem}
The proof of Theorem~\ref{T150} will be based on a fixed point argument and on the notion of Lagrangian solution, which are going to be introduced in the next subsection.

\subsection{Lagrangian solutions}
We start by proving an auxiliary well posedness result for and ODE in $Y$ of the form
\begin{equation}\label{T141}
\dot y_t=b_{\Psi_t}(y_t),\qquad y_0=\bar y,
\end{equation}
where $[0,T]\ni t\mapsto\Psi_t\in\cP_1(Y)$ is a given continuous curve and $\bar y\in Y$.
\begin{proposition}\label{T143}
Assume that for every $y\in Y$ and $\Psi\in\cP_1(Y)$ the velocity $v_\Psi\colon Y\to\R{d}$ satisfies (v1)-(v3) and the operator $\cT(x,\Psi)\colon \Lip(U)\to\Lip(U)$ satisfies (T0)-(T3). 
Let $\Psi\in C^0([0,T];(\cP_1(Y),\cW_1))$ and assume that there exists $R>0$ such that, in addition, $\Psi_t\in\cP(B_R^Y)$ for all $t\in[0,T]$.
Then, for every choice of $\bar y\in Y$, 
the ODE \eqref{T141} has a unique solution.
\end{proposition}
\begin{proof}
We set $b(t,y)\coloneqq b_{\Psi_t}(y)$ according to \eqref{T034}.
Since $t\mapsto\Psi_t$ is continuous, using \eqref{T117} we get that, for any fixed $y\in Y$, $b(\cdot,y)$ is continuous, which is condition (ii) in Corollary~\ref{cor:Brezis}.
Condition (iii) of Corollary~\ref{cor:Brezis} is a direct consequence of \eqref{T109}.
Furthermore, \eqref{T106} and \eqref{T110}  yield \eqref{eq:151} and \eqref{eq:152}, respectively. 
The proof is concluded by Corollary~\ref{cor:Brezis}.
\end{proof}
In view of the previous result, the following definition is justified.
\begin{defin}[Transition map]\label{T142}
The transition map $\bY_\Psi(t,s,\bar y)$ associated with the ODE \eqref{T141}, replacing the initial condition by $y_s=\bar y$, is defined through 
$$\bY_\Psi(t,s,\bar y)=y_t,$$
where $t\mapsto y_t$ is the unique solution to \eqref{T141} .
\end{defin}
We can now proceed to defining the notion of Lagrangian solution to \eqref{T140}.
\begin{defin}[Lagrangian solutions]\label{T038}
Let $\Lambda\in C^0([0,T];(\cP_1(Y),\cW_1))$ and let $\bar\Lambda\in\cP_c(Y)$ be a given initial datum. 
We say that $\Lambda$ is a Lagrangian solution to the initial value problem for the equation \eqref{T140} starting from $\bar\Lambda$ if and only if it satisfies the fixed point condition
\begin{equation}\label{T039}
\Lambda_t=\bY_\Lambda(t,0,\cdot)_{\#} \bar\Lambda\qquad\text{for every $0\leq t\leq T$},
\end{equation}
where $\bY_\Lambda(t,s,y)$ are the transition maps associated with the ODE \eqref{T141}.
\end{defin}
\begin{remark}\label{T155}
It follows from Definition~\ref{pushforward} that Lagrangian solutions are also Eulerian solutions.
\end{remark}
\begin{remark}\label{156}
For a fixed $N\in\N$, let $\Lambda_t^N$ be the empirical measures associated with the unique solution to \eqref{T132} with initial datum $\bar y^i$, $i=1,\ldots,N$. If we now set $\bar\Lambda^N\coloneqq \frac1N\sum_{i=1}^N \delta_{\bar y^i}$, by Definition~\ref{T142} there holds
\begin{equation}\label{T157}
\Lambda_t^N=\bY_{\Lambda^N}(t,0,\cdot)_{\#} \bar\Lambda^N\qquad\text{for every $0\leq t\leq T$}.
\end{equation}
Hence, $\Lambda^N$ is a Lagrangian and Eulerian solution to \eqref{T140} starting from $\bar\Lambda^N$.
\end{remark}
We now want to show that an infinite-dimensional converse of Proposition~\ref{T143} holds, proving that indeed, in our case, every Eulerian solution is also a Lagrangian solution.
This stems out of a general abstract principle known as the \emph{superposition principle}, in the version introduced in \cite{AFMS2018} (see also \cite[Theorem 8.2.1]{AGS2008} and \cite[Theorem 7.1]{AT2014}). In the statement below, the {\it evaluation map} $\ev_t$ is defined, at a given $t \in [0, T]$, by
\[
\textstyle \ev_t(\gamma)\coloneqq \gamma(t)\quad \hbox{for all }\gamma \in C([0, T]: E)\,. 
\]
\begin{theorem}[superposition principle]\label{superposition}
Let $(E,\|\cdot\|_E)$ be a separable Banach space, let $b\colon (0,T)\times E\to E$ be a Borel vector field and 
let $\mu_t\in\cP(E)$, $t\in [0,T]$, be a continuous curve with
\begin{equation}\label{eq:prix}
\int_0^T \int_E \|b_t\|_E\,\de\mu_t\de t<+\infty.
\end{equation}
If
$$\frac{\de}{\de t}\mu_t+\mathrm{div}({ b_t}\mu_t)=0$$
in duality with cylindrical functions $\phi\in C^1_b(E)$, precisely of the form (here $\langle\cdot,\cdot\rangle$ denotes the
duality map between $E$ and $E'$)
$$
\varphi(\langle y,z_1'\rangle,\langle y,z_2'\rangle,\ldots,\langle y,z_N'\rangle)
$$
with $\varphi\in C^1_b(\R{N})$ and $z_1',\ldots,z_N'\in E'$, then there exists $\eeta\in\cP(C ([0,T];E))$
concentrated on absolutely continuous solutions to the ODE $\dot{y}=b_t(y)$
and with $(\ev_t)_\# \eeta=\mu_t$ for all $t\in [0,T]$.
\end{theorem}
\begin{proof}
See \cite[Theorem~5.2]{AFMS2018}.
\end{proof}
Combining the abstract result Theorem~\ref{superposition} with the uniqueness granted by Proposition~\ref{T143}, we can prove the announced equivalence result. Notice that the proof has an intermediate step, since in order to apply Proposition~\ref{T143} we must first ensure that a Eulerian solution $\Lambda_t$ has (equi)compact support for all $t$. We are able to deduce this from Theorem~\ref{superposition} and the assumption that the initial datum $\bar\Lambda\in\cP_c(Y)$.
\begin{theorem}\label{equivalence}
Let $\Lambda\in C^0([0,T];(\cP_1(Y),\cW_1))$ and let $\bar\Lambda\in\cP_c(Y)$ be a given initial datum. 
Assume that $\Lambda$ is a Eulerian solution to the initial value problem for $\partial_t \Lambda_t+\div(b_{\Lambda_t}\Lambda_t)=0$ (see \eqref{T140}) starting from $\bar \Lambda$, in the sense of \eqref{T037}, then there exists $R>0$ such that $\Lambda_t\in\cP(B_R^Y)$ for all $t\in[0,T]$ and
\begin{equation*}
\Lambda_t=\bY_\Lambda(t,0,\cdot)_{\#} \bar\Lambda\qquad\text{for every $0\leq t\leq T$},
\end{equation*}
where $\bY_\Lambda(t,s,y)$ are the transition maps associated with the ODE \eqref{T141}.
\end{theorem}
\begin{proof}
Since $\Lambda\in C^0([0,T];(\cP_1(Y),\cW_1))$, the map
\begin{equation}
t\mapsto m_1(\Lambda_t)= \int_{Y} \lVert y \rVert_{\overline Y}\,\de\Lambda_t(y)
\end{equation}
is continuous, and hence bounded, in $[0,T]$.
Set $b_t(y)\coloneqq b_{\Lambda_t}(y)$, for $y\in Y$, and extend it to zero on $\overline Y\setminus Y$.
Using \eqref{T110}, we have
\begin{equation}\label{T161}
\begin{split}
\!\!\!\! \int_0^T \int_{\overline Y} & \lVert b_t\rVert_{\overline Y}\, \de\Lambda_t(y)\de t =  \int_0^T \int_{Y} \lVert b_t\rVert_{\overline Y}\,\de\Lambda_t(y)\de t \\
& \leq \int_0^T \int_{Y} M\big(1+ \lVert y \rVert_{\overline Y}+ m_1(\Lambda_t)\big)
\,\de\Lambda_t(y)\de t 
\leq  TM\Big(1+2\max_{t\in[0,T]} m_1(\Lambda_t)\Big) <+\infty.
\end{split}
\end{equation}
Hence, we can apply Theorem~\ref{superposition} with $E=\overline Y$ and $\mu_t=\Lambda_t$, obtaining that $\Lambda_t=(\ev_t)_\#\eeta$ for a suitable $\eeta\in\cP(C([0,T];\overline Y))$ concentrated on absolutely continuous solutions
to ODE
\begin{equation}\label{T141+}
\dot{y}=b_{\Lambda_t}(y)\qquad \text{in $[0,T]$.}
\end{equation}
Now, using \eqref{T110} again, we have
\begin{equation}\label{T162}
\lVert b_{\Lambda_t}(y) \rVert_{\overline Y} \le M\left(1+ \max_{t\in[0,T]}m_1(\Lambda_t)+ \lVert y \rVert_{\overline Y}\right) \le M_\Lambda \left(1+ \lVert y \rVert_{\overline Y}\right)\,, 
\end{equation}
where we set $M_\Lambda\coloneqq M\left(1+ \max_{t\in[0,T]}m_1(\Lambda_t)\right)$. The equality  $\bar\Lambda=\Lambda_0=(\ev_0)_\#\eeta$, which reads
\[
\int_{\cP(C([0,T];\overline Y))} \phi(\gamma(0))\,\mathrm{d}\eeta(\gamma)=\int_Y \phi(y)\,\mathrm{d}\bar\Lambda(y)
\]
for each $\phi \in C_b(Y)$, implies that $\eeta$ is concentrated on the set of solutions to \eqref{T141+} satisfying $y(0) \in B^Y_r$, where $r$ is such that $\mathrm{supp}\,\bar \Lambda \subset B^Y_r$. With \eqref{T162} and the Gr\"onwall inequality, each of these solutions must satisfy $y(t)\in  B^Y_{R}$, where $R$ is explictly given by
\[
R:=R_{r, M, \Lambda, T}=(r+M_\Lambda T)\mathrm{e}^{M_\Lambda T}\,.
\]
From the equality $\Lambda_t=(\ev_t)_\#\eeta$ we then deduce that $\Lambda_t \in \cP_1(B^Y_R)$ for all $t \in [0, T]$.
We can therefore apply Proposition~\ref{T143} and exploit the uniqueness of the solution to the Cauchy problem \eqref{T141} to deduce the representation
\[
\gamma(t)= \bY_\Lambda(t,0,\gamma(0))
\]
with $\gamma(0) \in   B^Y_r$, for each continuous path $\gamma \in \mathrm{supp}\,\eeta$. With the equality $\Lambda_t=(\ev_t)_\#\eeta$, this gives
\[
\int_Y \phi(y)\,\mathrm{d}\Lambda_t(y)= \int_{\cP(C([0,T];\overline Y))} \phi(\bY_\Lambda(t,0,\gamma(0)))\,\mathrm{d}\eeta(\gamma)=  \int_Y \phi(\bY_\Lambda(t,0,y))\,\mathrm{d}\bar\Lambda(y)
\]
for each $\phi \in C_b(Y)$, which implies the conclusion.
\end{proof}

\subsection{Proof of Theorem~\ref{T150}} As a preliminary step towards the proof, we need the following lemma, assuring that the size of the support of a Lagrangian solution in the sense of \eqref{T039} can be \emph{a priori} estimated from the data of the problem.
\begin{lemma}\label{T162}
Assume that for every $y\in Y$ and $\Psi\in\cP_1(Y)$ the velocity $v_\Psi\colon Y\to\R{d}$ satisfies (v1)-(v3) and the operator $\cT(x,\Psi)\colon \Lip(U)\to\Lip(U)$ satisfies (T0)-(T3).  Let $\Lambda\in C^0([0,T];(\cP_1(Y),\cW_1))$ and let $\bar\Lambda\in\cP_c(Y)$ be a given initial datum. Fix $r>0$ such that $\bar\Lambda$ has support in $B^Y_r$, and let $M$ be the constant given by \eqref{T110}.
Assume that $\Lambda$ is a Lagrangian solution to the initial value problem for the equation \eqref{T140} starting from $\bar\Lambda$ in the sense of \eqref{T039}. Then, for $R=(r+MT)\mathrm{e}^{2MT}$ we have
$$\Lambda_t \in \cP_1(B^Y_{R})\quad \hbox{for all }t\in [0, T]\,.$$
\end{lemma}

\begin{proof}
For $r$, $R$ as in the statement, it suffices to show that we have
\begin{equation}\label{T556}
\max_{y \in B_r^Y}\lVert \bY_{\Lambda^i} (t,0,y)\rVert_{\bar Y}\le R
\end{equation}
for all $t \in [0, T]$. Indeed, if this holds the statement immediately follows by \eqref{T039} and elementary properties of the push-forward measure, taking into account that $\bar\Lambda$ has support in $B^Y_r$.
To prove the above claim,  we first observe, that by definition of Lagrangian solutions and the fact that $\bar \Lambda \in \cP(B^Y_r)$ we immediately have
\begin{equation}\label{T564}
m_1(\Lambda_t)\le \max_{y \in B_r^Y}\lVert \bY_{\Lambda} (t,0,y)\rVert_{\bar Y}
\end{equation}
for all $t \in [0, T]$.  We now set $f(s)=\max_{y \in B_r^Y}\lVert \bY_{\Lambda} (s,0, y)\rVert_{\bar Y}$. Then, one has by definition  of the transition map, \eqref{T110}, and \eqref{T564} that for every choice of $y \in B^Y_r$
\[
\lVert \bY_{\Lambda} (t,0, y)\rVert_{\bar Y}\le r +M\int_0^t (1+\lVert \bY_{\Lambda} (s,0, y)\rVert_{\bar Y}+ m_1(\Lambda_s))\,\mathrm{d}s \le  r +M\int_0^t (1+2 f(s))\,\mathrm{d}s
\]
which implies by the Gronwall inequality $f(t)\le (r+Mt)\mathrm{e}^{2Mt}$ for all $t$, confirming \eqref{T556}.
\end{proof}

\begin{proof}[\it Proof of Theorem \ref{T150}]
The proof goes through a finite dimensional approximation and involves three steps.

{\it Step 1: stability of Lagrangian solutions}. We fix $r>0$, two initial data $\bar \Lambda^1$, $\bar \Lambda^2 \in \cP(B^Y_r)$ and assume that two Lagrangian solutions $\Lambda^1_t$, $\Lambda^2_t$, starting from $\bar \Lambda^1$,  and $\bar \Lambda^2$, respectively, exist. We fix $R=(r+MT)\mathrm{e}^{2MT}$ and the corresponding contant $L_R$ provided by \eqref{T106}-\eqref{T117}.
We claim that 
\begin{equation}\label{T163}
\mathcal{W}_1(\Lambda^1_t, \Lambda^2_t)\le \mathrm{e}^{L_R t+ \mathrm{e}^{L_R t}-1}\,\mathcal{W}_1(\bar \Lambda^1, \bar\Lambda^2)\quad\hbox{for all }t\in [0, T]\,.
\end{equation}
To prove this claim, we fix $y_1$ and $y_2 \in B^Y_r$ and observe that
\begin{equation}\label{T165}
\lVert \bY_{\Lambda^i} (t,0, y_i)\rVert_{\bar Y} \le R
\end{equation}
for all $t \in [0, T]$ and $i=1, 2$. This can be proved along similar lines as in the proof of \eqref{T564}. With \eqref{T165} and \eqref{T106}-\eqref{T117}, the solutions $y^1(t)$ and $y^2(t)$ to the ODE's $\dot y^i= b_{\Lambda^i}(y^i)$ with initial data $y^1$ and $y^2$ respectively, satisfy 
\begin{equation*}
\begin{split}
\frac{\mathrm{d}}{\mathrm{d} t}\lVert y^1-y^2\rVert(t)&\leq 
\lVert b_{\Lambda^1}(y^1(t))-b_{\Lambda^1}(y^2(t))\rVert+\lVert b_{\Lambda^1}(y^2(t))-b_{\Lambda^2}(y^2(t))\rVert
\\
&\leq
L_R\lVert y^1-y^2\rVert(t)+L_R \mathcal{W}_1(\Lambda_t^1,\Lambda_t^2)\,.
\end{split}
\end{equation*}
This gives, by means of a comparison argument, that 
\begin{equation*}
\lVert y^1(t)-y^2(t)\rVert\leq 
\mathrm e^{L_R t}\lVert y^1-y^2\rVert+L_R\int_0^t \mathrm e^{L_R(t-\tau)}\mathcal{W}_1(\Lambda_\tau^1,\Lambda_\tau^2)\,\mathrm{d}\tau\,;
\end{equation*}
equivalently,
\begin{equation}\label{T166}
\lVert\bY_{\Lambda^1} (t,0, y^1)-\bY_{\Lambda^2} (t,0, y^2)\rVert\leq 
\mathrm e^{L_R t}\lVert y^1-y^2\rVert+L_R\int_0^t \mathrm e^{L_R(t-\tau)}\mathcal{W}_1(\Lambda_\tau^1,\Lambda_\tau^2)\,\mathrm{d}\tau\,
\end{equation}
for alle $t \in [0, T]$ and $y_1$ and $y_2 \in B^Y_r$.

Now, let $\Pi$ be an optimal coupling between $\bar \Lambda^1$ and $\bar \Lambda^2$. Then clearly, by the definition of Lagrangian solutions,  $\left(\bY_{\Lambda^1}(t,0, y^1),\bY_{\Lambda^2}(t,0,y^2)\right)_\#{\Pi}$ 
 is a coupling between $\Lambda_t^1$ and $\Lambda_t^2$. Therefore 
\begin{equation*}
\begin{split}
\mathcal{W}_1(\Lambda_t^1,\Lambda_t^2)\leq{} & \int_{Y\times Y} \lVert\bY_{\Lambda^1}(t,0,y^1)-\bY_{\Lambda^2}(t,0,y^2)\rVert\,\mathrm{d}\Pi(y^1,y^2)\\
={}& \int_{B^Y_r\times B^Y_r} \lVert \bY_{\Lambda^1}(t,0,y^1)-\bY_{\Lambda^2}(t,0,y^2)\rVert\,\mathrm{d}\Pi(y^1,y^2)\,,
\end{split}
\end{equation*}
where we also used that $\bar \Lambda^1$, $\bar \Lambda^2 \in \cP(B^Y_r)$. Hence, using \eqref{T166} we get
\begin{equation*}
\begin{split}
\mathcal{W}_1(\Lambda_t^1,\Lambda_t^2)\leq{} & 
\mathrm e^{L_R t}\int_{B^Y_r\times B^Y_r} \lVert y^1-y^2\rVert\,\mathrm{d}\Pi(y^1,y^2)+ L_R\int_0^t\mathrm e^{L_R(t-\tau)}\mathcal{W}_1(\Lambda_{\tau}^1,\Lambda_{\tau}^2)\,\mathrm{d}\tau\\
={}& \mathrm e^{L_R t}\mathcal{W}_1(\bar \Lambda^1,\bar \Lambda^2)+L_R\int_0^t\mathrm e^{L_R(t-\tau)}\mathcal{W}_1(\Lambda_{\tau}^1,\Lambda_{\tau}^2)\,\mathrm{d}\tau\,.
\end{split}
\end{equation*}
With this and the Gr\"onwall inequality, we get \eqref{T163}.

{\it Step 2: existence and approximation of Lagrangian solutions}. We start by fixing a sequence of atomic measures $\bar \Lambda^N \in \cP(B^Y_r)$ such that 
\begin{equation}\label{T167}
\lim_{N\to \infty}\mathcal{W}_1(\bar \Lambda_N, \bar \Lambda)=0\,.
\end{equation}
Such a sequence can be for instance constructed as follows: choose $\bar y_i(\omega)\in Y$ independent and identically distributed, with law
$\bar\Lambda$, so that the random measures
$\bar\Lambda^N(\omega)\coloneqq\frac1N\sum_{i=1}^N\delta_{\bar y_i(\omega)}$ almost surely converge 
in $\cP_1(Y)$ to $\bar\Lambda$, and choose $\omega$ such that this happens. Now,  let $\Lambda_t^N$ be the empirical measures associated with the unique solution to \eqref{T132} with initial datum $\bar y^i$, $i=1,\ldots,N$. As noticed in  \eqref{T157}, $\Lambda_t^N$ are Lagrangian solutions to \eqref{T140} starting from  $\bar \Lambda^N$. Hence, \eqref{T163} provides a constant $C:=C(M, r, T)$ such that
\[
\mathcal{W}_1 (\Lambda_t^N, \Lambda_t^M)\le C \mathcal{W}_1(\bar \Lambda^N,  \bar\Lambda^M)
\] 
for all $t \in [0, T]$ and $N$, $M\in \mathbb{N}$. It follows that $\Lambda^N_t \in C([0, T]; (\cP_1(B^Y_R), \mathcal{W}_1))$ is a Cauchy sequence. Let then $\Lambda_t \in C([0, T]; (\cP_1(B^Y_R), \mathcal{W}_1))$ be the limit of the sequence  $\Lambda^N_t$. 

For a given $\bar y \in B^Y_r$, consider now the solutions $y^N(t)$ and $y(t)$ to the ODE's $\dot y^N= b_{\Lambda^N}(y^N)$, and $\dot y= b_{\Lambda}(y)$, respectively,  with initial datum $\bar y$. Let $R'\ge R$ be an upper bound\footnote{observe that at this point of the proof we cannot a priori exclude that $R'>R$, since we still do not know that $\Lambda_t$ is a Lagrangian solution, hence we cannot apply \eqref{T165} (which holds instead for $\Lambda^N$)} for $\max_{t\in [0, T]}\lVert y(t)\rVert_{\bar Y}$, which can be taken independent of $\bar y\in B^Y_r$. With  \eqref{T106}-\eqref{T117} we obtain
$$
\frac{\mathrm{d}}{\mathrm{d} t}\lVert y^N-y\rVert(t)\leq 
\lVert b_{\Lambda^N}(y^N)-b_{\Lambda}(y^N)\rVert+\lVert b_{\Lambda}(y^N)-b_{\Lambda}(y)\rVert\leq
L_{R'}\lVert y^N-y\rVert+L_{R} \mathcal{W}_1(\Lambda_t^N,\Lambda_t)\,.$$
Again by comparison, we deduce that
\[
\lVert\bY_{\Lambda^N} (t,0, \bar y)-\bY_{\Lambda} (t, 0, \bar y)\rVert\leq 
L_R\int_0^t \mathrm e^{L_{R'}(t-\tau)}\mathcal{W}_1(\Lambda_\tau^N,\Lambda_\tau)\,\mathrm{d}\tau\,
\]
which entails the uniform convergence of $\bY_{\Lambda^N} (\cdot, 0, \cdot)$ to $\bY_{\Lambda} (\cdot, 0, \cdot)$ in $[0, T]\times B^Y_r$. For each $t \in [0, T]$ this implies, together with \eqref{T167} and the fact that $\bY_{\Lambda} (t, 0, \cdot)$ is a Lipschitz map on $B_r^Y$, that
\[
\Lambda^N_t=\bY_{\Lambda^N}(t,0,\cdot)_{\#} \bar\Lambda^N \to \bY_{\Lambda}(t,0,\cdot)_{\#} \bar\Lambda
\]
in $\cP_1(Y)$, which gives \eqref{T039}.

{\it Step 3: uniqueness and conclusion}. Uniqueness of Lagrangian solutions, given the initial datum, follows now from \eqref{T163}. Existence and uniqueness of Eulerian solutions is now a consequence of Remark \ref{T155},  and Theorem \ref{equivalence}, respectively. The same argument used in the second step of this proof gives also part (ii) of the statement.
\end{proof}

\section{A leader-follower dynamics}\label{sec:LF}
We want to discuss the application of our results to the relevant scenario of two interacting populations, one consisting of ``leaders'' and the other one of ``followers'', with a switching rate between the two. In this setting, the set $U$ consists of two elements, that is $U:=\{F, L\}$ and is endowed with a two-valued distance
\[
0=\mathrm{dist }(F,F)=\mathrm{dist }(L,L)\,,\quad 1=\mathrm{dist }(F,L)=\mathrm{dist }(L,F)\,.
\]
The space $\mathrm{Lip}(\{F,L\})$ is a two-dimensional linear space spanned by the two indicator functions $\mathds{1}_F$ and $\mathds{1}_L$; accordingly, the space $\mathcal{F}(\{F,L\})$ is the two-dimensional space of signed Borel measures on the discrete set $\{F, L\}$, whose generic element $\xi$ is completely described by the two values $\xi_F\coloneqq\xi(\{F\})$ and $\xi_L\coloneqq\xi(\{L\})$.
A simple characterization of the operators $\mathcal{T}$ (and $\mathcal{T}^*$) complying with our set of assumptions (T0)-(T3) is then given in the following proposition.

\begin{proposition}\label{T175}
Let $U\coloneqq\{F, L\}$. Then $\mathcal{T}\colon \mathbb{R}^d \times \cP_1(Y)\to \mathrm{Lip}(\{F,L\})$ satisfies  (T0)-(T3)  if and only if there exist two functions $\alpha_F$, $\alpha_L \colon \mathbb{R}^d \times \cP_1(Y)\to [0, +\infty)$ such that
\begin{itemize}
\item[($\alpha$0)] for every $(x,\Psi)\in\R{d}\times\cP_1(Y)$ and $\lambda \in \cP_1(\{F,L\})$ it holds
\begin{equation*}
\begin{split}
(\cT^*(x,\Psi)\lambda)_F&=-\alpha_F (x,\Psi)\lambda_F + \alpha_L(x,\Psi)(1-\lambda_F)\,,\\
(\cT^*(x,\Psi)\lambda)_L&=\alpha_F (x,\Psi)\lambda_F - \alpha_L(x,\Psi)(1-\lambda_F);
\end{split}
\end{equation*}
\item[($\alpha$1)] for every $(x,\Psi)\in \R{d}\times\cP_1(Y)$, there exists $M_\cT>0$ such that the 
\begin{equation*}
0\le \alpha_\bullet(x,\Psi)\leq M_\cT\big(1+\lvert x\rvert+m_1(\Psi)\big),\qquad\text{for $\bullet=F,L$;}
\end{equation*}
\item[($\alpha$2)]for every $R>0$ there exists $L_{\cT,R}>0$ such that, for every $(x^1,\Psi^1),(x^2,\Psi^2)\in B_R\times\cP(B_R^Y)$, 
\begin{equation*}
|\alpha_\bullet(x^1,\Psi^1)-\alpha_\bullet(x^2,\Psi^2)|\leq L_{\cT,R}\big(|x^1-x^2|+\cW_1(\Psi^1,\Psi^2)\big),\qquad\text{for $\bullet=F,L$.}
\end{equation*}
\end{itemize}
\end{proposition}

\begin{proof}
Since $\mathrm{Lip}(\{F,L\})$ is a two-dimensional linear space, we can identify $\cT(x, \Psi)$ with its matrix representation with respect to the canonical basis $\{\mathds{1}_F,\,\mathds{1}_L\}$ and endow the space $\mathcal{L}(\mathrm{Lip}(\{F,L\}),\mathrm{Lip}(\{F,L\})$ with the Frobenius norm of such a matrix representation. Accordingly, the transpose matrix will be the representation of $\cT^*(x, \Psi)$  with respect to the canonical basis of $\mathcal{F}(\{F, L\})$ consisting of the two Dirac masses $\delta_F$ and $\delta_L$.

Now, condition (T0) is  equivalent to the fact that
\begin{equation}\label{T168}
\cT(x, \Psi)=\left(
\begin{array}{c}
\!\!-\alpha_F(x, \Psi) \quad \alpha_F(x, \Psi)\\
\alpha_L(x, \Psi)  \quad -\alpha_L(x, \Psi)
\end{array}
\right)
\end{equation}
with $\alpha_F$, $\alpha_L \colon \mathbb{R}^d \times \cP_1(Y)\to [0, +\infty)$. In turn, this is equivalent to ($\alpha$0) by a direct computation. Again by a direct computation,  (T2) is equivalent to ($\alpha$2), while (T1) is equivalent to the inequalities
\[
| \alpha_\bullet(x,\Psi)|\leq M_\cT\left(1+\lvert x\rvert+m_1(\Psi)\right)\,,\qquad\text{for $\bullet=F,L$.}
\]
In particular, $\alpha_F$ and $\alpha_L$ are uniformly bounded for $(x,\Psi)\in B_R\times\cP_1(B^Y_R)$, hence (T3) holds if and only if the nondiagonal elements $\alpha_F(x, \Psi)$ and $\alpha_L(x, \Psi)$ of the matrix representation \eqref{T168} are nonnegative. This implies that (T1) and (T3) together are equivalent to ($\alpha$1), which concludes the proof.
\end{proof}

To each element $\Psi$ of $\cP_1(\mathbb{R}^d \times \cP_1(\{F, L\}))$ we can associate a followers' and a leaders' distribution in a natural way, as we are going to discuss in the next definition.

\begin{defin}\label{T911}
Let $\Psi \in \cP_1(\mathbb{R}^d \times \cP_1(\{F, L\}))$ The followers distribution $\mu^F_\Psi$ associated to $\Psi$ is the positive Borel measure on $\mathbb{R}^d$ defined by
\begin{subequations}\label{169}
\begin{equation}\label{T169}
\mu^F_\Psi(B):=\int_{B \times \cP_1(\{F, L\})} \lambda_F \,\mathrm{d}\Psi(x, \lambda)
\end{equation}
for each Borel set $B \subset \mathbb{R}^d$. Similarly, the leaders distribution $\mu^L_\Psi$ associated to $\Psi$ is the positive Borel measure on $\mathbb{R}^d$ defined by
\begin{equation}\label{T170}
\mu^L_\Psi(B):=\int_{B \times \cP_1(\{F, L\})} (1-\lambda_F) \,\mathrm{d}\Psi(x, \lambda)\,.
\end{equation}
\end{subequations}
\end{defin}

Both the measures defined above have a simple interpretation. For instance, $\mu^F_\Psi(B)$ is the expected value of the number of leaders in a region $B$ for a probability distribution $\Psi$ on $\mathbb{R}^d \times \cP_1(\{F, L\})$. We also observe that the sum of the two measures $\mu^F_\Psi$ and $\mu^L_\Psi$ is exactly the $x$-marginal of $\Psi(x, \lambda)$. From a practical point of view, the integrals appearing in \eqref{169} 
can be computed identifying the metric space $(\cP_1(\{F, L\}), \mathcal{W}_1)$ with the $1$-dimensional symplex $[0,1]$ endowed with the Euclidean distance (this is indeed an isometry).
We also point out the following inequalities, whose proof is straightforward
\begin{equation}\label{T172}
\begin{split}
m_1(\mu^F_\Psi)+m_1(\mu^L_\Psi)&\le m_1(\Psi)\,,\quad \hbox{for all }\Psi \in   \cP_1(\mathbb{R}^d \times \cP_1(\{F, L\}));\\
\lVert \mu_{\Psi_1}^F-\mu_{\Psi_2}^F\rVert_{BL} &\le 2 \mathcal{W}_1(\Psi_1,  \Psi_2) \quad \hbox{for all }\Psi_1, \Psi_2 \in   \cP_1(\mathbb{R}^d \times \cP_1(\{F, L\}))\\
 \lVert \mu_{\Psi_1}^L- \mu_{\Psi_2}^L\rVert_{BL}&\le 2 \mathcal{W}_1(\Psi_1,  \Psi_2) \quad \hbox{for all }\Psi_1, \Psi_2 \in   \cP_1(\mathbb{R}^d \times \cP_1(\{F, L\}))\,.
\end{split}
\end{equation}

Using the previous definition, we can also provide some relevant examples of velocity fields $v_\Psi$ complying with assumptions (v1)-(v3).
\begin{proposition}
Let $U=\{F, L\}$. For $\Psi \in \cP_1(\mathbb{R}^d \times \cP_1(\{F, L\}))$, consider the velocity field
\begin{equation}\label{T171}
v_{\Psi}(x, \lambda):=\lambda_F \left(K^{FF}\star \mu^F_\Psi+K^{LF}\star \mu^L_\Psi\right)+(1-\lambda_F) \left(K^{FL}\star \mu^F_\Psi+K^{LL}\star \mu^L_\Psi\right)
\end{equation}
where $\mu^F_\Psi$ and $\mu^L_\Psi$ are defined in \eqref{169} 
and the interaction kernels $K^{ij}\colon \mathbb{R}^d \to \mathbb{R}^d$ satisfy
\begin{equation*}
\begin{split}
|K^{ij}(x)|&\le M(1+|x|),\quad \mbox{for all }x \in \mathbb{R}^d;
\\
|K^{ij}(x_1)-K^{ij}(x_2)|&\le L_R\,|x_1-x_2|,\quad \mbox{for all }x_1, x_2 \in B_R\,.
\end{split}
\end{equation*}
Then, $v_{\Psi}(x, \lambda)$ satisfies (v1)-(v3).
\end{proposition}
\begin{proof}
The result follows by a direct computation.
\end{proof}

\begin{remark}\label{T916}
The velocity field \eqref{T171} corresponds to a particle model where each follower experiences a velocity $K^{FF}\star \mu^F_\Psi+K^{LF}\star \mu^L_\Psi$, which combines the action of the overall followers and leaders distribution. Similarly, each leader is under the action of the velocity field $K^{FL}\star \mu^F_\Psi+K^{LL}\star \mu^L_\Psi$. Hence, \eqref{T171} is an \emph{average} velocity of the system, weigthed by the probability $\lambda$ that a particle at $x$ has of being a leader or a follower.
\end{remark}

In a similar spirit to the previous proposition, also  transition rates $\alpha_F$ and $\alpha_L$ depending on $\mu^F_\Psi$ and $\mu^L_\Psi$ can be considered in our setting. This is for instance the point of view taken in \cite[Assumption (H4) and Appendix A]{ABRS2018}, where also some explicit examples were provided. Our assumptions are actually more general than those considered there: in particular, we can allow for an explicit dependence on the space variable $x$.
\begin{proposition}\label{T918}
Let $U=\{F, L\}$. Consider two functions $\alpha_F$, $\alpha_L\colon \mathbb{R}^d \times \mathcal{M}_+( \mathbb{R}^d) \times \mathcal{M}_+( \mathbb{R}^d)\to [0, +\infty)$ satisfying the following assumptions:
\begin{itemize}
\item there exists a constant $M$ such that, for $\bullet=F,L$,
\begin{equation}\label{T173}
0\le \alpha_\bullet(x,\mu, \nu)\le M\left(1+|x|+m_1(\mu)+m_1(\nu)\right)
\end{equation}
for all $ x\in \mathbb{R}^d$ and $(\mu, \nu) \in \mathcal{M}_+( \mathbb{R}^d) \times \mathcal{M}_+( \mathbb{R}^d)$;
\item for all $R>0$, there exist a constant $L_R$ such that, for $\bullet=F,L$,
\begin{equation}\label{T174}
|\alpha_\bullet(x_1,\mu_1, \nu_1)-\alpha_\bullet(x_2,\mu_2, \nu_2)|\le L_R\left(|x_1-x_2|+\lVert \mu_1-\mu_2\rVert_{BL}+ \lVert \nu_1-\nu_2\rVert_{BL}\right)
\end{equation}
for all $x_1$, $x_2 \in B_R$ and $(\mu_1, \nu_1)$,  $(\mu_2, \nu_2) \in \mathcal{M}_+( B_R) \times \mathcal{M}_+( B_R)$.
\end{itemize}
For $\Psi \in \cP_1(\mathbb{R}^d \times \cP_1(\{F, L\}))$, define $\mu^F_\Psi$ and $\mu^L_\Psi$ as in \eqref{169}. 
Then, the functions
\begin{equation}\label{T177}
\alpha_F(x, \Psi):=\alpha_F(x, \mu^F_\Psi, \mu^L_\Psi) \quad \mbox{and} \quad  \alpha_L(x, \Psi):=\alpha_L(x, \mu^F_\Psi, \mu^L_\Psi)
\end{equation}
satisfy Assumptions ($\alpha$0)-($\alpha2$) in Proposition \ref{T175}.
\end{proposition}
\begin{proof} 
The result follows from \eqref{T173}-\eqref{T174} by means of the inequalities in \eqref{T172}.
\end{proof}

The main result of this Section is an existence and uniqueness result for the system of equations considered in \cite{ABRS2018}, which  we are going to deduce from Theorem \ref{T150}. By doing this, we will extend the result in \cite[Proposition 3.2]{ABRS2018} to the case were the transition rates $\alpha_F$, $\alpha_L$ are allowed to explicitly depend on $x$, which was not considered there. The equations we consider are namely
\begin{align}\label{eq:macroleadfollstrong}
\left\{\begin{aligned}
\partial_t \mu^F_t& = -\mathrm{div}\big((K^{F}\star\mu^F_t + K^{L}\star\mu^L_t)\mu^F_t\big) - \alpha_F(x,\mu^F_t,\mu^L_t)\mu^F_t + \alpha_L(x, \mu^F_t,\mu^L_t)\mu^L_t,\\
\partial_t \mu^L_t& = -\mathrm{div}\big((K^{F}\star\mu^F_t + K^{L}\star\mu^L_t)\mu^L_t\big) +\alpha_F(x, \mu^F_t,\mu^L_t)\mu^F_t -\alpha_L(x, \mu^F_t,\mu^L_t)\mu^L_t
\end{aligned}\right.
\end{align}
to be solved by two positive Borel measures $\mu^F$ and $\mu^L$ attaining, for $t=0$, an initial datum $\bar{\mu}^F$, and $\bar{\mu}^L$, respectively, which we assume to have compact support in $\mathbb{R}^d$. As customary in this kind of models, we will assume that the initial total population is normalized to $1$, i.e.
\[
\bar{\mu}^F(\mathbb{R}^d)+\bar{\mu}^L(\mathbb{R}^d)=1\,.
\]

\begin{remark}\label{Tfin}
The velocity field in \eqref{eq:macroleadfollstrong} corresponds to the choice $K^{FF}=K^{FL}=K^F$ and  $K^{LF}=K^{LL}=K^L$ in \eqref{T171}. The key observation is that, in this case, the field $v_\Psi(x, \lambda)$ shows no explicit dependence on $\lambda$ and is simply given by
\begin{equation}\label{T176}
v_\Psi(x)=K^{F}\star\mu^F_\Psi + K^{L}\star\mu^L_\Psi\,.
\end{equation}
This will eventually allow us to decouple equation \eqref{T140} into the simpler system \eqref{eq:macroleadfollstrong}. Such an analysis is not possible if more than two different kernels are considered in \eqref{T171}. In that general case, the mean-field limit of the associated particle system must be formulated in terms of a solution $\Lambda$ to  \eqref{T140}, defined in the product space $\mathbb{R}^d \times \cP_1(\{F, L\})$.
\end{remark}

To proceed to the announced result, we need to recall te definition of a solution to \eqref{eq:macroleadfollstrong} which has been considered in \cite{ABRS2018}. Below, the shortcut $ \mathcal{M}_c(\mathbb{R}^d)$ is used to denote a positive Borel measure having compact support in $\mathbb{R}^d$.

\begin{defin}[Solution of system \eqref{eq:macroleadfollstrong}]\label{def:solmacroleadfoll}
	Let $(\overline{\mu}^F,\overline{\mu}^L)\in\mathcal{M}_c(\mathbb{R}^d)\times \mathcal{M}_c(\mathbb{R}^d)$ be given, as well as $\mu^F,\mu^L:[0,T]\rightarrow\mathcal{M}_c(\mathbb{R}^d)$. We say that the couple $(\mu^F_t,\mu^L_t)$ is a solution of system \eqref{eq:macroleadfollstrong} with initial datum $(\overline{\mu}^F,\overline{\mu}^L)$ when
	\begin{enumerate}
		\item $\mu^F_0 = \overline{\mu}^F$ and $\mu^L_0 = \overline{\mu}^L$;
		\item for each $i\in\{F,L\}$, the function $t\to \mu^i_t$ is continuous with respect to the topology of weak convergence of measures;
		\item there exists $R_T > 0$ such that $\bigcup_{t \in [0,T]}\supp(\mu^i_t)\subseteq B_{R_T}$ for every $i \in \{F,L\}$;
		\item for every $\varphi \in \mathcal{C}^1_c(\mathbb{R}^d)$ and $i \in \{F,L\}$ it holds
		\begin{align*}\begin{split}
		\frac{d}{dt}\int_{\mathbb{R}^d}\varphi(x)d\mu^i_t(x)& = \int_{\mathbb{R}^d}\nabla\varphi(x)\cdot\left[\sum_{j\in\{F,L\}}(K^{j}\star\mu^j_t)(x)\right]d\mu^i_t(x) \\
		&\;\;\;-\int_{\mathbb{R}^d}\varphi(x)\,\alpha_i(x, \mu^F_t,\mu^L_t)\mathrm{d}\mu^i_t(x)+\int_{\mathbb{R}^d}\varphi(x)\,\alpha_{\neg i}(x,\mu^F_t,\mu^L_t)\mathrm{d}\mu^{\neg i}_t(x),
		\end{split}
		\end{align*}
		for almost every $t\in[0,T]$, with
		\begin{align*}
		\neg i \coloneqq \begin{cases}
		L & \text{ if } i = F,\\
		F & \text{ if } i = L.
		\end{cases}
		\end{align*}
	\end{enumerate}
\end{defin}

We can now state the existence and uniqueness result for system \eqref{eq:macroleadfollstrong}.

\begin{theorem}\label{T187}
Let $U=\{F, L\}$. Consider two functions $\alpha_F$, $\alpha_L\colon \mathbb{R}^d \times \mathcal{M}_+( \mathbb{R}^d) \times \mathcal{M}_+( \mathbb{R}^d)\to [0, +\infty)$ satisfying \eqref{T173}-\eqref{T174} and two kernels $K^F$, $K^L\colon \mathbb{R}^d \to \mathbb{R}^d$ with 
\begin{equation}\label{T186}
\begin{split}
|K^{F}(x)|+|K^{L}(x)|&\le M(1+|x|)\quad \mbox{for all }x \in \mathbb{R}^d;
\\
|K^{F}(x_1)-K^{F}(x_2)|+|K^{L}(x_1)-K^{L}(x_2)|&\le L_R\,|x_1-x_2|\quad \mbox{for all }x_1, x_2 \in B_R\,.
\end{split}
\end{equation}
For $\Psi \in \cP_1(\mathbb{R}^d \times \cP_1(\{F, L\}))$, define $\mu^F_\Psi$ and $\mu^L_\Psi$ as in \eqref{169}. 
For $x \in \mathbb{R}^d$ and $\Psi \in \cP_1(\mathbb{R}^d \times \cP_1(\{F, L\}))$, let  $v_\Psi(x)$  and $\mathcal{T}(x, \Psi)$ be given by \eqref{T176} and \eqref{T168} 
respectively, and consider the corresponding velocity field $b_\Psi$ as in \eqref{T034}.

Then, if $\Lambda \in C([0, T];\cP_1(\mathbb{R}^d \times \cP_1(\{F, L\}), \mathcal{W}_1)$ is the unique solution to \eqref{T140} starting from $\overline{\Lambda}\in \cP_c(\mathbb{R}^d \times \cP_1(\{F, L\})$, the measures $\mu^F_t:=\mu^F_{\Lambda_t}$ and $\mu^L_t:=\mu^L_{\Lambda_t}$ are the unique solutions to \eqref{eq:macroleadfollstrong} with initial data $\bar \mu^F=\mu^F_{\bar\Lambda}$ and $\bar \mu^L=\mu^L_{\bar\Lambda}$.
\end{theorem}

\begin{proof}
We start by observing that, under our assumptions on $\alpha_F$, $\alpha_L$, $K^F$ and $K^L$, the previous results of this section assure that the field $b_\Psi$ complies with the requirements of Theorem \ref{T150}, hence existence and uniqueness of the solution $\Lambda$ to \eqref{T140} starting from $\bar\Lambda$ is guaranteed. We split the proof into two parts, proving first existence and then uniqueness of the solutions.

{\it Existence.}  For $\Lambda$ as above, define $\mu^F_t:=\mu^F_{\Lambda_t}$ and $\mu^L_t:=\mu^L_{\Lambda_t}$. By definition and Lemma \ref{T162}, conditions (i) and (iii) in Definition \ref{def:solmacroleadfoll} are satisfied. The continuity property (ii) is instead a direct consequence of \eqref{T172} and the continuity of $\Lambda$ as a function of the time.
We therefore only have to check (iv) in Definition \ref{def:solmacroleadfoll}. We perform the required computation only for $\mu^F_t$, since the one for $\mu^F_t$ follows along similar lines. We take $\varphi  \in C^1_c(\mathbb{R}^d)$ and we define, for all $(x, \lambda) \in \mathbb{R}^d \times \mathcal{F}(\{F, L\})$ the test function
\[
\phi(x, \lambda)=\lambda_F \varphi(x)\,.
\]
We notice that the $(\mathbb{R}^d)^*$-component of the  Fr\'echet differential of $\phi$ at $(x, \lambda)$ is given by $\lambda_F D\varphi(x)$, while the action of the other component is independent of $\lambda$ by linearity and is given by
\[
\xi\mapsto \varphi(x) \xi_F\,.
\]
We apply the definition \eqref{T037} of Eulerian solution to the above test function $\phi$, which does not depend on $t$, and we get
\begin{equation}\label{T037+}
\begin{split}
\frac{\mathrm{d}}{\mathrm{d}t}\int_{\mathbb{R}^d \times \cP_1(\{F, L\})}\lambda_F\varphi(x)\,\de\Lambda_t(x, \lambda)= 
\int_{\mathbb{R}^d \times \cP_1(\{F, L\})}\lambda_F\nabla \varphi(x)\cdot v_{\Lambda_t}(x)\,\de\Lambda_t(x, \lambda)\\
+\int_{\mathbb{R}^d \times \cP_1(\{F, L\})}\varphi(x)\left(\mathcal{T}^*(x,\Lambda_t)\lambda\right)_F\,\de\Lambda_t(x, \lambda)\,,
\end{split}
\end{equation}
for all $t \in [0, T]$. We observe now that \eqref{T169}-\eqref{T170} are equivalent to the duality relationships
\begin{equation}\label{T589}
\begin{split}
\int_{\mathbb{R}^d}\zeta(x)\,\mathrm{d}\mu^F_\Psi(x)=\int_{\mathbb{R}^d \times \cP_1(\{F, L\})} \lambda_F\, \zeta(x)\,\mathrm{d}\Psi(x, \lambda)\,,\\ \int_{\mathbb{R}^d}\zeta(x)\,\mathrm{d}\mu^L_\Psi(x)=\int_{\mathbb{R}^d \times \cP_1(\{F, L\})} (1-\lambda_F) \zeta(x)\,\mathrm{d}\Psi(x, \lambda)
\end{split}
\end{equation}
for all $\zeta \in C_b(\mathbb{R}^d)$. Applying the first one to the functions $\varphi(x)$ and $\nabla \varphi(x)\cdot v_{\Lambda_t}(x)\in C_b(\mathbb{R}^d)$ gives\footnote{here it is crucial that the velocity field $v_{\Lambda_t}(x)$ only depends on $x$.}
\begin{equation}\label{T177}
\begin{split}
\int_{\mathbb{R}^d \times \cP_1(\{F, L\})}\lambda_F\varphi(x)\,\de\Lambda_t(x, \lambda)&=\int_{\mathbb{R}^d} \varphi(x)\,\mathrm{d}\mu^F_t(x)\\
\int_{\mathbb{R}^d \times \cP_1(\{F, L\})}\lambda_F\nabla \varphi(x)\cdot v_{\Lambda_t}(x)\,\de\Lambda_t(x, \lambda)&=\int_{\mathbb{R}^d}\nabla \varphi(x)\cdot v_{\Lambda_t}(x)\,\de\mu^F_t(x)\,.
\end{split}
\end{equation}
With Proposition~\ref{T175}, \eqref{T177}, and applying \eqref{T589} 
to the functions $-\varphi(x)\alpha_F(x, \mu^F_t, \mu^L_t)$, and $\varphi(x)\alpha_L(x, \mu^F_t, \mu^L_t)$, respectively, we get
\[
\begin{split}
\int_{\mathbb{R}^d \times \cP_1(\{F, L\})}\varphi(x)\left(\mathcal{T}^*(x,\Lambda_t)\lambda\right)_F\,\de\Lambda_t(x, \lambda)=\\
-\int_{\mathbb{R}^d}\varphi(x)\,\alpha_F(x, \mu^F_t,\mu^L_t)\mathrm{d}\mu^F_t(x)+\int_{\mathbb{R}^d}\varphi(x)\,\alpha_{L}(x,\mu^F_t,\mu^L_t)\mathrm{d}\mu^{L}_t(x)\,.
\end{split}
\]
Hence, also using \eqref{T177} and the explicit expression \eqref{T176} of the field $v_{\Lambda_t}(x)$, \eqref{T037+} is equivalent to equality (iv) in Definition \ref{def:solmacroleadfoll} for $i=F$, as required. This proves existence of a solution to \eqref{eq:macroleadfollstrong} .

{\it Uniqueness}. The proof is divided into two steps.

{\it Uniqueness- Step 1: continuous dependence for an auxiliary equation}. For a given $Z>0$ we fix the class of Carath\'edory vector fields
\[
\mathcal{V}_Z:=\{v \in L^\infty([0, T], \mathrm{Lip}(\mathbb{R}^d)):  \lVert v \rVert_{L^\infty([0, T], \mathrm{Lip}(\mathbb{R}^d))} \le Z  \}
\]
and we denote as usual by $\mathbf{Y}_v(t, s, \cdot)$ the associated transition maps, which satisfy the equalities
\begin{equation}\label{T179}
\mathbf{Y}_v(s,s,x)=x, \quad \frac{\de}{\de t}\mathbf{Y}_v(t,s,x)=v_t(\mathbf{Y}_v(t,s,x))
\end{equation}
for every $0\le s\le t\le T$ and $x\in \mathbb{R}^d$. From these, we can deduce the Gr\"onwall-type estimates
\begin{equation}\label{T180}
\begin{split}
|\mathbf{Y}_v(t,s,x_1)-\mathbf{Y}_v(t,s,x_2)|&\le \mathrm{e}^{Z(t-s)}|x_1-x_2|,\\
|\mathbf{Y}_v(t,s,x)-\mathbf{Y}_w(t,s,x)|&\le (t-s)\mathrm{e}^{Z(t-s)}\sup_{s\in [0, t]}\lVert w_s-v_s \rVert_{C_b(\mathbb{R}^d)}
\end{split}
\end{equation}
for every $0\le s\le t\le T$,   $x$, $x_1$, $x_2\in \mathbb{R}^d$ and $v, w \in \mathcal{V}_Z$.
For $v \in \mathcal{V}_Z$ and a given narrowly continuous family $\xi_t$ of signed measures satisfying $|\xi_t| \le R$ for all $t \in [0, T]$, we consider the inhomogeneous equation
\begin{equation}\label{T182}
\partial_t \mu_t + \mathrm{div}(v_t \mu_t)=\xi_t\,.
\end{equation}
We claim that, for a given initial datum $\bar \mu \in \mathcal{M}(\mathbb{R}^d)$, the unique solution of \eqref{T182} starting from $\bar \mu$ is given by the variation of constants formula
\begin{equation}\label{T183}
\mu_t=\mathbf{Y}_v(t, 0, \cdot)_\#\bar \mu+\int_0^t \mathbf{Y}_v(t, s, \cdot)_\#\xi_s\,\de s\,.
\end{equation}
A direct computation using \eqref{T179} proves indeed that the above formula provides a solution to \eqref{T182}, while uniqueness follows by taking the difference of two solutions and using that the comparison principle \cite[Proposition 8.1.7]{AGS2008} for the continuity equation.  

Now, if $\sigma_t$ is another narrowly continous family of signed measures, using \eqref{T180} for all test functions $\varphi$ with $|\varphi|\le 1$ and  $\mathrm{Lip}(\varphi) \le 1$ we have
\[
\begin{split}
&\left|\int_0^t \int_{\mathbb{R}^d} \varphi(\mathbf{Y}_v(t,s,x))\de(\xi_s-\sigma_s)\,\de s\right|\le \\
 &\left( \int_0^t \mathrm{e}^{Z(t-s)}\de s\right)\sup_{s\in [0, t]}\lVert \xi_s-\sigma_s \rVert_{BL}\le t\, C_{T, Z}\sup_{s\in [0, t]}\lVert \xi_s-\sigma_s \rVert_{BL}\,.
\end{split}
\]
With this, \eqref{T180}, and \eqref{T183}, we also deduce that, if $\nu$ solves \eqref{T182} for another velocity field $w \in \mathcal{V}_z$, another narrowly continuous family $\sigma_t$ with $|\sigma_t|\le R$,  and  the same initial datum $\bar \mu$, the following estimate holds true:
\begin{equation}\label{T184}
\lVert \mu_t-\nu_t\rVert_{BL} \le t \, C_{\bar \mu, T, R, Z}\left(\sup_{s\in [0, t]}\lVert w_s-v_s \rVert_{C_b(\mathbb{R}^d)}+\sup_{s\in [0, t]}\lVert \xi_s-\sigma_s \rVert_{BL}\right)\,,
\end{equation}
where the constant   $C_{\bar \mu, T, R, Z}$ is given by
\[
 C_{\bar \mu, T, R, Z}:=|\bar \mu|\mathrm{e}^{ZT}+ (1+RT)C_{T, Z}\,.
\]

{\it Uniqueness- Step 2: conclusion}. Consider two solutions $(\mu^F_t, \mu^L_t)$ and $(\nu^F_t, \nu^L_t)$ of \eqref{eq:macroleadfollstrong} starting from the same initial datum $(\bar \mu^F, \bar \mu^L)$. Observe that the equation preserves the total mass, that is
\[
\mu^F_t(\mathbb{R}^d)+ \mu^L_t(\mathbb{R}^d)=\nu^F_t(\mathbb{R}^d)+\nu^L_t(\mathbb{R}^d)=1
\]
for all $t \in [0, T]$. This can be rewritten as
\begin{equation}\label{T185}
|\mu^F_t|+ |\mu^L_t|=|\nu^F_t|+|\nu^L_t|=1\,,
\end{equation}
since all the above measures are positive. Now, on the one hand $\mu^F_t$ solves  \eqref{T182} with $v_t=K^F\star \mu^F_t+K^L\star \mu^L_t$ and $\xi_t=-\alpha_F(\cdot,\mu^F_t, \mu^L_t)\mu^F_t+\alpha_L(\cdot,\mu^F_t, \mu^L_t)\mu^L_t$. On the other hand, $\nu^F_t$ solves  \eqref{T182} with $v_t$ replaced by $w_t=K^F\star \nu^F_t+K^L\star \nu^L_t$ and $\sigma_t=-\alpha_F(\cdot,\nu^F_t, \nu^L_t)\nu^F_t+\alpha_L(\cdot,\nu^F_t, \nu^L_t)\nu^L_t$ in place of $\xi_t$.Since $\mu^F_t$, $\mu^L_t$, $\nu^F_t$, and $\nu^L_t$ have compact support contained in a fixed ball $B_{R_T}$ by Definition \ref{def:solmacroleadfoll}, we can assume (up to multiplying $v_t$ and $w_t$ by a suitable cut-off function  not affecting equation \eqref{T182}) that $v$ belongs to $\mathcal{V}_Z$ for a constant $Z$ only depending on $T$, $R_T$ and the constant $M$ in \eqref{T186}. Using \eqref{T173}, \eqref{T174}, and \eqref{T185}, we also get the estimates
\[
\begin{split}
&|\xi_t|\le 2M(1+2R_T),\quad |\sigma_t|\le 2M(1+2R_T),\\
\lVert\sigma_t-\xi_t\|_{BL}&\le (M(1+2R_T)+2L_{R_T})\left(\lVert\mu^F_t-\nu^F_t\rVert_{BL}+\lVert\mu^L_t-\nu^L_t\rVert_{BL}\right).
\end{split}
\]
In a similar way, using \eqref{T186} we obtain
\[
\lVert v_t-w_t\rVert_{C_b(\mathbb{R}^d)} \le (M(1+2R_T)+L_{2R_T})\left(\lVert\mu^F_t-\nu^F_t\rVert_{BL}+\lVert\mu^L_t-\nu^L_t\rVert_{BL}\right).
\]
Combining the previous inequalities with \eqref{T184} and \eqref{T185}, we get that there exists a constant $C$, only depending on $T$, $R_T$ and $M$, such that 
\[
\lVert \mu^F_{\tau}-\nu^F_{\tau}\rVert_{BL} \le \tau \, C\left(\sup_{s\in [0, \tau]}\lVert \mu^F_{\tau}-\nu^F_\tau \rVert_{BL}+\sup_{s\in [0, \tau]}\lVert \mu^L_{\tau}-\nu^L_{\tau}\rVert_{BL}\right)\,.
\]
for all $0\le \tau \le t \le T$. A similar estimate also holds for $\lVert \mu^L_{\tau}-\nu^L_{\tau}\rVert_{BL}$. Taking the supremum in the left-hand sides and summing up the resulting inequalities we obtain
\[
\sup_{s\in [0, t]}\lVert \mu^F_{t}-\nu^F_t\rVert_{BL}+\sup_{s\in [0, t]}\lVert \mu^L_{t}-\nu^L_{t}\rVert_{BL} \le 2t \, C\left(\sup_{s\in [0, t]}\lVert \mu^F_{t}-\nu^F_t\rVert_{BL}+\sup_{s\in [0, t]}\lVert \mu^L_{t}-\nu^L_{t}\rVert_{BL}\right)\,.
\]
for all $0 \le t \le T$. This implies that $\mu^F_t=\nu^F_t$, as well as $\mu^L_t=\nu^L_t$ for all $0 \le t \le \frac{1}{2C}$. As $C$ only depends on $T$, $R_T$ and $M$, iterating the argument a finite number of times yields uniqueness on all $[0, T]$.
\end{proof}

\begin{remark}
For all choices of the inital data $\bar \mu^F$ and $\bar \mu^L$ whose sum is a probability measure $\bar \mu$ with compact support, we can construct $\bar \Lambda \in \cP_c(Y)$ so that $\mu^F_{\bar \Lambda}=\bar \mu^F$ and $\mu^L_{\bar \Lambda}=\bar \mu^L$. This can be done, for instance, as follows: if $g^F(x)$ is the Radon-Nikodym derivative of $\bar \mu^F$ with respect to $\bar \mu$, we can define, for $\bar \mu$-a.e. $x\in \mathbb{R}^d$, the measure $\lambda_x\in \cP_1(\{F, L\})$ via $\lambda_x:= g^F(x) \delta_F + (1-g^F(x))\delta_L$. Then, the measure $\bar \Lambda$, defined by duality through
\[
\int_{Y} \phi (y)\,\de \bar \Lambda:=\int_{\mathbb{R}^d} \phi(x, \lambda_x)\,\de \bar \mu(x)
\]
for all $\phi \in C_b(Y)$, satisfies the claim. Hence,  Theorem \ref{T187} is a full existence result for system \eqref{eq:macroleadfollstrong}.
\end{remark}

To conclude the Section, we observe that combining the above Remark with Theorems \ref{T150} and \ref{T187} we obtain a mean-field derivation of system \eqref{eq:macroleadfollstrong} which extends the results of \cite[Section 4]{ABRS2018}.

\section{Further applications}\label{sec:other}
In this section we present two situations which fit in the theory presented in Section~\ref{sec:abstract_mod} and are suitable to describe possible applications.
In Section~\ref{discreteU} we present the generalization of the results in Section~\ref{sec:LF} to the case of a discrete and finite label space $U$; in Section~\ref{continuousU} we focus on a continuum of labels and compare it with the results of \cite{AFMS2018}.

\subsection{Discrete and finite spaces of labels $U$}\label{discreteU}
In this case, we will identify the discrete space of labels $U=\{u_1,\ldots,u_H\}$ with the set of the indices of the labels, so that our model will be $U=\{1,\ldots,H\}$.
We will endow $U$ with the Euclidean distance (restricted to $U$), namely,
\begin{equation}\label{T900}
\dist{h}{k}=|h-k|, \qquad\text{for $h,k\in U$.}
\end{equation}
Then the space $\Lip(U)$ is an $H$-dimensional space spanned by the indicator functions $\mathds{1}_h$, for $h\in\{1,\ldots,N\}$.
Consequently, the free space $\cF(U)$ is the space of signed Borel measures on $U$ whose generic element $\xi$ is characterized by the values $\xi_h \coloneqq \xi(\{h\})$.

Analogously to Proposition~\ref{T175}, we have the following characterization of the operators $\cT$ (and $\cT^*$) satisfying assumptions (T0)-(T3).
\begin{proposition}\label{T901}
Let $U=\{1,\ldots,H\}$. 
Then $\cT\colon \R{d}\times\cP_1(U)\to\Lip(U)$ satisfies (T0)-(T3) if and only if there exist $H^2$ functions $\alpha_{hk}\colon\R{d}\times\cP_1(U)\to[0,+\infty)$ such that
\begin{itemize}
\item[($\bar\alpha$0)] for every $(x,\Psi)\in\R{d}\times\cP_1(Y)$ and $\lambda \in \cP_1(U)$ it holds
\begin{equation*}
\begin{split}
(\cT^*(x,\Psi)\lambda)_h&=-\alpha_{hh} (x,\Psi)\lambda_h + \sum_{k\neq h}\alpha_{kh}(x,\Psi)\lambda_k\,,\qquad\text{for all $h\in U$}
\end{split}
\end{equation*}
with
\begin{equation}\label{T903}
\alpha_{hh}(x,\Psi)=\sum_{k\neq h} \alpha_{hk}(x,\Psi),\qquad\text{for all $h\in U$};
\end{equation}
\item[($\bar\alpha$1)] for every $(x,\Psi)\in \R{d}\times\cP_1(Y)$, there exists $M_{\cT}>0$
such that 
\begin{equation*}
0\le \alpha_{hk}(x,\Psi)\leq M_{\cT}\big(1+\lvert x\rvert+m_1(\Psi)\big),\qquad\text{for all $h,k\in U$;} 
\end{equation*}
\item[($\bar\alpha$2)]for every $R>0$ there exists $L_{\cT,R}>0$ 
such that, for every $(x^1,\Psi^1),(x^2,\Psi^2)\in B_R\times\cP(B_R^Y)$, 
\begin{equation*}
|\alpha_{hk}(x^1,\Psi^1)-\alpha_{hk}(x^2,\Psi^2)|\leq L_{\cT,R}\big(|x^1-x^2|+\cW_1(\Psi^1,\Psi^2)\big),\qquad\text{for $h,k\in U$.} 
\end{equation*}
\end{itemize}
\end{proposition}
\begin{proof}
The proof is analogous to that of Proposition~\ref{T175}.
\end{proof}
We notice that a matrix representation of the operator $\cT$ analogous to that in \eqref{T168} holds
\begin{equation}\label{T902}
\cT(x,\Psi)=\left(
\begin{array}{cccc}
-\alpha_{11}(x,\Psi) & \alpha_{12}(x,\Psi) & \cdots & \alpha_{1H}(x,\Psi)  \\
\alpha_{21}(x,\Psi) & -\alpha_{22}(x,\Psi) & \cdots & \alpha_{2H}(x,\Psi)  \\
\vdots & \vdots & \ddots & \vdots  \\
\alpha_{H1}(x,\Psi) & \alpha_{H2}(x,\Psi) & \cdots & -\alpha_{HH}(x,\Psi) 
\end{array}\right),
\end{equation}
where, as sa consequence of \eqref{T903}, the sum of the elements on each row must give zero.
In this discrete setting, the operator $\cT^*(x,\psi)$ has a matrix representation given by the inverse matrix of that representing $\cT(x,\Psi)$, so that 
\begin{equation}\label{T904}
\cT^*(x,\Psi)=\left(
\begin{array}{ccccc}
-\alpha_{11}(x,\Psi) & \alpha_{21}(x,\Psi) & \cdots & \alpha_{H1}(x,\Psi)  \\
\alpha_{12}(x,\Psi) & -\alpha_{22}(x,\Psi) & \cdots & \alpha_{H2}(x,\Psi)  \\
\vdots & \vdots & \ddots & \vdots  \\
\alpha_{1H}(x,\Psi) & \alpha_{2H}(x,\Psi) & \cdots & -\alpha_{HH}(x,\Psi) 
\end{array}\right)
\end{equation}
and condition \eqref{T903} implies that 
the sum of the elements on each column must vanish.

Definition~\ref{T911} is adapted to the following
\begin{defin}\label{T912}
Let $\Psi \in \cP_1(\mathbb{R}^d \times \cP_1(U))$ The distribution $\mu^h_\Psi$ of the agents with label $h\in U$ associated to $\Psi$ is the positive Borel measure on $\mathbb{R}^d$ defined by
\begin{equation}\label{T913}
\mu^h_\Psi(B):=\int_{B \times \cP_1(U)} \lambda_h \,\mathrm{d}\Psi(x, \lambda)
\end{equation}
for each Borel set $B \subset \mathbb{R}^d$. 
\end{defin}

Also in this case, upon choosing suitable interaction kernels, the measures $\mu_\Psi^h$ defined in \eqref{T913} can be used to construct velocity fields (of the type in \eqref{T171}) $v_\Psi$ satisfying (v1)-(v3).
\begin{proposition}\label{T914}
Let $U=\{1,\ldots,H\}$. 
For $\Psi \in \cP_1(\mathbb{R}^d \times \cP_1(U))$, consider the velocity field
\begin{equation}\label{T915}
v_{\Psi}(x, \lambda)\coloneqq \sum_{h,k\in U} \lambda_k \left(K^{hk}\star \mu^h_\Psi\right)
\end{equation}
where $\mu^h_\Psi$, for $h\in U$, are defined in \eqref{T913} 
and the interaction kernels $K^{hk}\colon \mathbb{R}^d \to \mathbb{R}^d$, for $h,k\in U$, satisfy
\begin{equation*}
\begin{split}
|K^{hk}(x)|&\le M(1+|x|),\quad \mbox{for all }x \in \mathbb{R}^d;
\\
|K^{hk}(x_1)-K^{hk}(x_2)|&\le L_R\,|x_1-x_2|,\quad \mbox{for all }x_1, x_2 \in B_R\,.
\end{split}
\end{equation*}
Then, $v_{\Psi}(x, \lambda)$ satisfies (v1)-(v3).
\end{proposition}
\begin{proof}
The result follows by a direct computation.
\end{proof}
Remark~\ref{T916} applies in this case as well, so that the velocity field $v_\Psi$ defined in \eqref{T914} can be interpreted as an average velocity of the system, weighted by the probability that a particle at $x$ has of having label $k$.

Proposition~\ref{T918} concerning transition rates $\alpha_h$'s depending on the $\mu_\Psi^h$'s and explicitly on the space variable $x$ is generalized to the following.
\begin{proposition}\label{T919}
Let $U=\{1,\ldots,H\}$. 
Consider functions $\alpha_{hk}\colon \mathbb{R}^d \times \big(\mathcal{M}_+( \mathbb{R}^d)\big)^H
\to [0, +\infty)$, for $h,k\in U$, satisfying the following assumptions:
\begin{itemize}
\item there exists a constant $M$ such that, for all $h,k\in U$, 
\begin{equation}\label{T920}
0\leq \alpha_{hk}(x,\mu^1,\ldots, \mu^H)\leq M\bigg(1+|x|+\sum_{l\in U}m_1(\mu^l)\bigg)
\end{equation}
for all $ x\in \mathbb{R}^d$ and $(\mu^1,\ldots,\mu^H) \in \big(\mathcal{M}_+( \mathbb{R}^d)\big)^H$;
\item for all $R>0$, there exist a constant $L_R$ such that, for all $h,k\in U$,
\begin{equation}\label{T921}
|\alpha_{hk}(x_1,\mu_1^1,\ldots,\mu_1^H)-\alpha_{hk}(x_2,\mu_2^1,\ldots,\mu_2^H)|\le L_R\bigg(|x_1-x_2|+\sum_{l\in U}\lVert \mu_1^l-\mu_2^l\rVert_{BL}\bigg)
\end{equation}
for all $x_1$, $x_2 \in B_R$ and $(\mu_1^1,\ldots,\mu_1^H),(\mu_2^1,\ldots, \mu_2^H) \in \big(\mathcal{M}_+( B_R)\big)^H$.
\end{itemize}
For $\Psi \in \cP_1(\mathbb{R}^d \times \cP_1(U))$, define $\mu^h_\Psi$ as in \eqref{T913}, for $h\in U$. 
Then, the functions
\begin{equation}\label{T922}
\alpha_{hk}(x, \Psi):=\alpha_{hk}(x, \mu^1_\Psi,\ldots, \mu^H_\Psi), \quad \mbox{for }h,k\in U
\end{equation}
satisfy Assumptions ($\bar\alpha$0)-($\bar\alpha2$) in Proposition \ref{T901}.
\end{proposition}
\begin{proof} 
The result follows from \eqref{T920} and \eqref{T921} by means of the following  inequalities
\begin{equation}\label{T172a}
\begin{split}
\sum_{l\in U}m_1(\mu^l_\Psi)&\le m_1(\Psi)\,,\quad \hbox{for all }\Psi \in   \cP_1(\mathbb{R}^d \times \cP_1(U)),\\
\lVert \mu_{\Psi_1}^h-\mu_{\Psi_2}^h\rVert_{BL} &\le 2 \mathcal{W}_1(\Psi_1,  \Psi_2) \quad \hbox{for all }\Psi_1, \Psi_2 \in   \cP_1(\mathbb{R}^d \times \cP_1(U)),\;h\in U,
\end{split}
\end{equation}
which are a straightforward generalization of those in \eqref{T172}.
\end{proof}

Theorem~\ref{T187} can be generalized to the case of a finite discrete space of labels $U=\{1,\ldots,H\}$.
To obtain system \eqref{eq:macroleadfollstrong} in the current context, we have to assume that the interaction kernels $K^{hk}\colon\R{d}\to\R{d}$, for $h,k\in U$ introduced in Proposition~\ref{T914} are such that 
\begin{equation}\label{T923}
K^{hk}=K^h,\qquad\text{for all $h\in U$,}
\end{equation}
for $H$ kernels $K^h\colon \R{d}\to\R{d}$. In this case, analogously to the case with two labels, the velocity field $v_\Psi$ defined in \eqref{T915} does not depend on $\lambda$ anymore and has the form
\begin{equation}\label{T924}
v_\Psi(x)=\sum_{h\in U} K^h\star \mu_\Psi^h,
\end{equation}
where the $\mu_\Psi^h$'s are defined in \eqref{T913}.
Then, system \eqref{eq:macroleadfollstrong} becomes a set of $H$ equations
\begin{equation}\label{T925}
\partial_t \mu_t^h = -\div\bigg(\bigg(\sum_{k\in U} K^k\star \mu_t^k\bigg)\mu_t^h\bigg)-\alpha_{hh}(x,\mu_t^1,\ldots,\mu_t^H)\mu_t^h + \sum_{k\neq h}\alpha_{kh}(x,\mu_t^1,\ldots,\mu_t^H)\mu_t^k,
\end{equation}
to be solved for Borel measures $\mu_h$ such that, at the initial time $t=0$, $\mu_0^h=\bar\mu^h$, for $h\in U$, where $\bar\mu^1,\ldots,\bar\mu^H$ are given Borel measures satisfying
\begin{equation}\label{T926}
\sum_{h\in U} \bar\mu^h(\R{d})=1.
\end{equation}
We give the following
\begin{defin}[Solution to system \eqref{T925}]\label{T927}
	Let $(\bar{\mu}^1,\ldots,\bar{\mu}^H)\in\big(\mathcal{M}_c(\mathbb{R}^d)\big)^H$ be given such that \eqref{T926} is satisfied, as well as $\mu^1,\ldots,\mu^H\colon [0,T]\rightarrow\mathcal{M}_c(\mathbb{R}^d)$. We say that $(\mu^1_t,\ldots,\mu^H_t)$ is a solution to system \eqref{T925} with initial datum $(\bar{\mu}^1,\ldots,\bar{\mu}^H)$ when
	\begin{enumerate}
		\item $\mu^h_0 = \bar{\mu}^h$ for each $h\in U$;
		\item for each $h\in U$, the function $t\to \mu^h_t$ is continuous with respect to the topology of weak convergence of measures;
		\item there exists $R_T > 0$ such that $\bigcup_{t \in [0,T]}\supp(\mu^h_t)\subseteq B_{R_T}$ for every $h\in U$;
		\item for every $\varphi \in \mathcal{C}^1_c(\mathbb{R}^d)$ and $h \in U$ it holds
		\begin{align*}\begin{split}
		\frac{\de}{\de t}\int_{\mathbb{R}^d} & \varphi(x)\,\de\mu^h_t(x) = \int_{\mathbb{R}^d}\nabla\varphi(x)\cdot\bigg(\sum_{k\in U}(K^{k}\star\mu^k_t)(x)\bigg)\de\mu^h_t(x) \\
		&-\int_{\mathbb{R}^d}\varphi(x)\alpha_{hh}(x, \mu^1_t,\ldots,\mu^H_t)\,\mathrm{d}\mu^h_t(x)+\int_{\mathbb{R}^d}\varphi(x)\sum_{k\neq h}\alpha_{kh}(x,\mu^1_t,\ldots,\mu^H_t)\,\mathrm{d}\mu^{k}_t(x),
		\end{split}
		\end{align*}
		for almost every $t\in[0,T]$. 
	\end{enumerate}
\end{defin}
We are now in a position to prove the most important result of this section, namely how a solution to \eqref{T925} follows from Theorem~\ref{T150}.
\begin{theorem}\label{T928}
Let $U=\{1,\ldots,H\}$. 
Consider  functions $\alpha_{hk}\colon \mathbb{R}^d \times \big(\mathcal{M}_+( \mathbb{R}^d)\big)^H\to [0, +\infty)$ satisfying \eqref{T920} and \eqref{T921} and $H$ kernels $K^h\colon \mathbb{R}^d \to \mathbb{R}^d$ with 
\begin{equation}\label{T186}
\begin{split}
\sum_{h\in U} |K^{h}(x)|&\le M(1+|x|),\quad \mbox{for all }x \in \mathbb{R}^d;
\\
\sum_{h\in U} |K^{h}(x_1)-K^{h}(x_2)|&\le L_R\,|x_1-x_2|,\quad \mbox{for all }x_1, x_2 \in B_R\,.
\end{split}
\end{equation}
For $\Psi \in \cP_1(\mathbb{R}^d \times \cP_1(U))$, define the $\mu^h_\Psi$'s as in \eqref{T913}. 
For $x \in \mathbb{R}^d$ and $\Psi \in \cP_1(\mathbb{R}^d \times \cP_1(U))$, let  $v_\Psi(x)$  and $\mathcal{T}(x, \Psi)$ be given by \eqref{T924}, and \eqref{T902}, 
respectively, and consider the corresponding velocity field $b_\Psi$ as in \eqref{T034}.

Then, if $\Lambda \in C([0, T];\cP_1(\mathbb{R}^d \times \cP_1(U), \mathcal{W}_1)$ is the unique solution to \eqref{T140} starting from $\overline{\Lambda}\in \cP_c(\mathbb{R}^d \times \cP_1(U)$, the measures $\mu^h_t:=\mu^h_{\Lambda_t}$, for $h\in U$, 
are the unique solutions to \eqref{T925} with initial data $\bar \mu^h=\mu^h_{\bar\Lambda}$, for $h\in U$.
\end{theorem}
\begin{proof}
The proof is analogous to that of Theorem~\ref{T187}
\end{proof}

It is useful to recall the notion of $Q$-matrix from the literature of Markov chains (see, e.g., \cite[Chapter~2]{Norris1997}).
Let $Q$ be a $H\times H$ matrix; the element $q_{hk}$ represents the transition rate from state $h$ to $k$.
Such a matrix is called $Q$-matrix satisfies the following conditions
\begin{itemize}
\item[(Q1)] $q_{hk}\geq0$, for all $h,k\in U, h \neq k$;
\item[(Q2)] $\sum_{k\in U} q_{hk}=0$, for all $h\in U$.
\end{itemize}
It is customary, in the literature on Markov chains, to complement these conditions by
\begin{itemize}
\item[(Q0)] $0\leq-q_{hh}<+\infty$, for all $h\in U$,
\end{itemize}
even though condition (Q0) is a consequence of (Q1) and (Q2).
\begin{remark}\label{T910}
It is easy to see that the matrix in \eqref{T902}, representing the operator $\cT(x,\Psi)$ for all $(x,\Psi)\in \R{d}\times\cP_1(Y)$ is a $Q$-matrix: indeed, condition \eqref{T903} is equivalent to (Q2) and ($\bar\alpha$1) is equivalent to (Q1).
\end{remark}
For $(x,\Psi)\in\R{d}\times\cP_1(\R{d}\times\cP_1(U))$, and for fixed initial conditions $(x^0,\lambda^0)\in Y$, the dynamics described by the vector field $b_\Psi$ defined in \eqref{T034} is 
\begin{equation}\label{T930}
\vec{\dot x}{\dot\lambda}=b_\Psi(x,\lambda)=\vec{v_\Psi(x)}{\cT^*(x,\Psi)\lambda},\qquad\vec{x}{\lambda}(0)=\vec{x^0}{\lambda^0}.
\end{equation}
Concerning the second equation, $\dot\lambda=\cT^*(x,\Psi)\lambda$, this is a linear equation in the space of measures, which can be easily integrated to give
\begin{equation}\label{T931}
\lambda(t)=\lambda_t=\cS_t(x,\Psi)\lambda^0,\qquad\text{with}\quad \cS_t(x,\Psi)\coloneqq e^{t\cT(x,\Psi)}.
\end{equation}
A classical result in the theory of Markov chains, \cite[Theorem~2.1.2]{Norris1997} grants that the fact that $\cT(x,\Psi)$ satisfies (Q1)-(Q2) is equivalent to the matrix $\cS_t(x,\Psi)$ being a \emph{stochastic matrix} for all times $t\geq0$, namely a matrix $S\in\R{H\times H}$ that satisfies
\begin{itemize}
\item[(S1)] $0\leq S_{kh}<+\infty$ for all $h,k\in U$;
\item[(S2)] $\sum_{h\in U} S_{kh}=1$ for all $k\in U$.
\end{itemize}
We remark that the explicit solution given by \eqref{T931} is available only for fixed $(x,\Psi)$, so that the matrix \eqref{T902} is constant and its exponential can be computed. 
This would restrict the dynamics to constant solutions $x(t)=x^0$ for all $t$ and to a constant distribution $\Psi\in\cP(Y)$.

\subsection{A continuum of labels}\label{continuousU}
We now turn to the case in which $U$ is a compact metric space, as in the general theory developed in Section~\ref{sec:abstract_mod}, and that it is a continuum.
We will shortly present some possible expressions of the velocity field $b_\Psi$ defined in \eqref{T034} in which both components feature a two-player game that determines the evolution.

Let $\bar y\in Y$ and let $\cK\colon Y\times Y\to \R{d}\times\cF(U)$ be an interaction kernel such that, for $\Psi\in\cP(Y)$,
\begin{equation}\label{T932}
b_\Psi(y)=\int_Y \cK(y,y')\,\de\Psi(y').
\end{equation}
Our aim is to study system \eqref{T141} where the evolution is driven by the field $b_\Psi$ defined in \eqref{T932}.
Recalling \eqref{T034}, let  
$\cK_x\colon Y\times Y\to\R{d}$ and $\cK^*_\lambda\colon Y\times Y\to\cF(U)$ be the components of $\cK$. Then 
\eqref{T932} gives
\begin{equation}\label{T933}
v_\Psi(y)=\int_Y \cK_x(y,y')\,\de\Psi(y'),\qquad\text{and}\qquad \cT^*(x,\Psi)\lambda=\int_Y \cK^*_\lambda(y,y')\,\de\Psi(y').
\end{equation}
We furthermore assume that each of $\cK_x$ and $\cK^*_\lambda$ accounts for an averaging over all the strategies in $U$.
To this aim, let $V\colon (\R{d}\times U)^2\to\R{d}$ be such that
\begin{equation}\label{T030x}
\cK_x(y,y')=\int_U\int_U V(x,u,x',u')\,\de\lambda'(u')\de\lambda(u)
\end{equation}
The interpretation of the field $V$ is the following: the quantity 
$$V(x,u,x',u')\in\R{d}$$
gives the direction that the payer at $x$ should follow if it has label/strategy $u$ when playing a game against a player at $x'$ with label/strategy $u'$.
Keeping the structure \eqref{T141}, \eqref{T933}, and \eqref{T030x} in mind, the full law of motion for the variable $x$ is given by
\begin{equation}\label{T011}
\dot x=v_\Psi(y)= \int_U\int_Y\int_U V(x,u,x',u')\,\de\lambda'(u')\de\Psi(x',\lambda')\de\lambda(u),
\end{equation}
which can be interpreted in the following way: the player at $x$ with label/strategy $u$ plays a game with all the other players at $x'$ with their labels/strategies $u'$ and the velocity $\dot x_t$ is determined by averaging over all the possible strategies of the opponent (the integral with respect to $\de\lambda'(u')$), over all the possible distributions of opponents with their distributions of strategies (the integral with respect to $\de\Lambda(x',\lambda')$), and finally over all the possible strategies that $x$ have at their disposal (the integral with respect to $\de\lambda(u)$).

For $(x,\Psi)\in\R{d}\times\cP_1(U)$ define the operator $\cT(x,\Psi)\colon \Lip(U)\to\Lip(U)$ by
\begin{equation}\label{T951}
(\cT(x,\Psi)f)(u)\coloneqq \int_Y\int_U J(x,u,x',u')f(u')\,\de\lambda'(u')\de\Psi(x',\lambda')-g(x,\Psi)(u)f(u),
\end{equation}
where $J\colon (\R{d}\times U)^2\to \R{}$, and the function $g(x,\Psi)\in C(U)$ can be interpreted as a \emph{global departure rate} from $u$.
\begin{proposition}\label{T950}
Let $\Psi\in\cP_1(Y)$,
let $V\colon (\R{d}\times U)^2\to\R{d}$, $J\colon (\R{d}\times U)^2\to \R{}$, and $g(\cdot,\Psi)\in C(U)$ be such that
\begin{itemize}
\item[(V1)] $V$ is locally Lipschitz with respect to all of its variables;
\item[(V2)] $V$ is sublinear in the spatial variables, namely there exists $C_V>0$ such that
$$|V(x,u,x',u')|\leq C_V(1+\lvert x\rvert+\lvert x'\rvert),\qquad\text{for all $(x,u),(x'u')\in(\R{d}\times U)^2$;}$$
\item[(J1)] $J$ is locally Lipschitz with respect to all of its variables;
\item[(J2)] $J$ is sublinear in the spatial variables, namely there exists $C_J>0$ such that
$$|J(x,u,x',u')|\leq C_J(1+\lvert x\rvert+\lvert x'\rvert),\qquad\text{for all $(x,u),(x'u')\in(\R{d}\times U)^2$.}$$
\end{itemize}
Then, the field $v_\Psi\colon Y\to\R{d}$ given by\eqref{T933} via \eqref{T030x} satisfies (v1)-(v3) and the operator $\cT(x,\Psi)\colon\Lip(U)\to\Lip(U)$ given by \eqref{T951} satisfies (T1)-(T3).
If, moreover,
\begin{itemize}
\item[(g)] for $(x,\Psi)\in\R{d}\times\cP_1(Y)$, 
\begin{equation}\label{T952}
g(x,\Psi)(u)=\int_Y\int_U J(x,u,x',u')\,\de\lambda'(u')\de\Psi(x',\lambda'),
\end{equation}
\end{itemize}
then (T0) is also satisfied.
In this case, the adjoint operator $\cT^*(x,\Psi)\colon\cF(U)\to\cF(U)$ is represented as in \eqref{T933}, by means of $\cK^*_\lambda\colon Y\times Y\to\cF(U)$ such that, for $f\in\Lip(U)$, 
\begin{equation}\label{T954}
\begin{split}
\langle \cK^*_\lambda(y,y'),f\rangle 
= & \int_U\int_U J(x,u,x',u')[f(u')-f(u)]\,\de\lambda'(u')\de\lambda(u). 
\end{split}
\end{equation}
\end{proposition}
\begin{proof}
Let $R>0$ be fixed and let $L_{V,R}$ be the Lipschitz constant of $V$ on $(B_R\times U)^2$. 
Properties (v1)-(v2) are obtained from (V1) by standard estimates keeping the definition of $\BL$ norm and the structure \eqref{T933}-\eqref{T030x} into account; the constant $L_{v,R}$ of (v1)-(v2) is determined by $L_{V,R}$ and $\diam U$.
Property (v3) follows from (V1) and (V2) and the constant $M_v$ of (v3) is determined by $C_V$.

Analogously, let $L_{J,R}$ be the Lipschitz constant of $J$ on $(B_R\times U)^2$.
Property (T2) is obtained from (J1) by standard estimates keeping the definition of $\BL$ norm and the structure \eqref{T951} into account; the constant $L_{\cT,R}$ of (T2) is determined by $L_{J,R}$ and $\diam U$.
Property (T1) is obtained also using (J2) and the constant $M_\cT$ of (T1) is determined by $C_J$.
Property (T3) follows from the boundedness of $J(\cdot,\cdot,x',u')$ and $g(\cdot,\Psi)$ on $B_R\times U$ and $U$, respectively.

Finally, a simple computation shows that (g) and (T0) are equivalent, so that, using \eqref{T952}, we obtain the following expression for \eqref{T951}
\begin{equation}\label{T953}
(\cT(x,\Psi)f)(u)= \int_Y\int_U J(x,u,x',u')[f(u')-f(u)]\,\de\lambda'(u')\de\Psi(x',\lambda'),\quad\text{for $f\in \Lip(U)$}. 
\end{equation}
Recalling, from \eqref{T034} and \eqref{T141}, that $\dot\lambda=\cT^*(x,\Psi)\lambda$, for $f\in\Lip(U)$ we have, using \eqref{T933},
\begin{equation*}
\begin{split}
\langle\dot\lambda,f\rangle = & \langle \cT^*(x,\Psi)\lambda,f \rangle = \langle \lambda, \cT(x,\Psi)f \rangle \\
=& \bigg\langle \lambda,\int_Y\int_U J(x,u,x',u')[f(u')-f(u)]\,\de\lambda'(u')\de\Psi(x',\lambda') \bigg\rangle \\
=& \int_U\int_Y\int_U J(x,u,x',u')[f(u')-f(u)]\,\de\lambda'(u')\de\Psi(x',\lambda')\de\lambda(u) \\
=& \int_Y \bigg[\int_U\int_U J(x,u,x',u')[f(u')-f(u)]\,\de\lambda'(u')\de\lambda(u)\bigg]\de\Psi(x',\lambda'),
\end{split}
\end{equation*}
which is \eqref{T954}.
The proof is complete.
\end{proof}

\begin{remark}\label{differences}
Let us assume that the hypotheses of Proposition~\ref{T950} hold;  the ODE in \eqref{T141}, written component-wise, reads
\begin{equation}\label{T955}
\vec{\dot x}{\dot\lambda}=\vec{v_\Psi(y)}{\cT^*(x,\Psi)\lambda},
\end{equation}
for $y=(x,\lambda)\in Y=\R{d}\times\cP(U)$ and $\Psi\in\cP_1(Y)$.

A comparison with \cite{AFMS2018} is in order. In \cite[equations (1.3), (1.4), and (1.8)]{AFMS2018} the ODE studied is (keeping the notations of \cite{AFMS2018})
\begin{equation}\label{T956}
\vec{\dot x}{\dot\sigma}=\vec{a(y)}{\Delta_{\Sigma,y}\sigma},
\end{equation}
where $y=(x,\sigma)\in C\coloneqq\R{d}\times\cP(U)$ and $\Sigma\in\cP_1(C)$.
The right-hand sides are
\begin{equation}\label{T957}
\begin{split}
a(y) = & \int_U e(x,u)\,\de\sigma(u), \\
\Delta_{\Sigma,y}\sigma = & \bigg(\int_C\int_U J(x,u,x',u')\,\de\sigma'(u')\de\Sigma(x',\sigma') \\
&-\int_U \int_C\int_U J(x,w,x',u')\,\de\sigma'(u')\de\Sigma(x',\sigma')\de\sigma(w)\bigg)\sigma,
\end{split}
\end{equation}
where $e\colon\R{d}\times U\to\R{d}$ and $J\colon(\R{d}\times U)^2\to\R{}$. 
The equation for $\sigma$ in \eqref{T956} is known in the literature of evolutionary games as \emph{replicator equation} (see, e.g., \cite{HS1998}).

In comparison with \eqref{T957}, we notice that the theory developed in Section~\ref{sec:abstract_mod} allows us to deal with a broader class of velocity fields, that include agent interaction. Indeed, the velocity field $a\colon C\to\R{d}$ in \eqref{T957} is linear in the mixed strategy $\sigma$ and depends only on the position $x$ of the agent; it does not depend neither on the global distribution $\Sigma$ of the system nor on any interaction between the agents.
The velocity field $v_\Psi\colon Y\to\R{d}$ in \eqref{T011}, instead, additionally takes into account the global distribution $\Psi$ of the system as well as the interaction between players through the kernel $V$.

The equations for the second component of $y$, namely $\lambda$ and $\sigma$ in \eqref{T955} and \eqref{T956}, respectively, share even fewer features. If, on the one hand, they both depend on the global distribution of the system $\Psi$ or $\Sigma$, on the other hand the second equation in \eqref{T955} is linear in $\lambda$ (see \eqref{T954}), whereas the replicator equation is quadratic in $\sigma$, through the triple integral in \eqref{T957}.
\end{remark}

\bigskip


\noindent {\bf Acknowledgments.} 
The authors thank the hospitality of the departments of mathematics of the Politecnico di Torino, of the Universit\`a di Napoli, and of the E.~Schr\"{o}dinger Institute in Vienna, where this research was developed.

The authors are members of the Gruppo Nazionale per l'Analisi Matematica, la Probabi\-lit\`a e le loro Applicazioni (GNAMPA) of the Istituto Nazionale di Alta Matematica (INdAM).

The work of FS is part of the project ``Variational methods for stationary and evolution problems with singularities and interfaces'' PRIN 2017 financed by the Italian Ministry of Education, University, and Research.

\end{document}